\documentclass[11pt]{amsart}
\usepackage{amsmath,amsthm,amsfonts,amssymb,amstext}
\theoremstyle{plain}

\usepackage{amssymb}
\usepackage{amscd}

\newtheorem{theorem}{Theorem}[section]

\newtheorem{lemma}{Lemma}[section]

\newtheorem{prop}{Proposition}[section]

\newtheorem{example}{Example}[section]

\theoremstyle{remark}
\newtheorem{remark}[theorem]{Remark}

\def\E{\operatorname{E}}
\def\F{\operatorname{F}}
\def\G{\operatorname{G}}

\def\I{\operatorname{I}}

\def\U{\operatorname{U}}

\def\SL{\operatorname{SL}}
\def\GL{\operatorname{GL}}

\def\SO{\operatorname{SO}}
\def\Sp{\operatorname{Sp}}
\def\SU{\operatorname{SU}}

\def\GSp{\operatorname{GSp}}

\def\PSU{\operatorname{PSU}}

\def\Ad{\operatorname{Ad}}

\def\Aut{\operatorname{Aut}}

\def\diag{\operatorname{diag}}

\def\det{\operatorname{det}}
\def\dim{\operatorname{dim}}

\def\exp{\operatorname{exp}}

\def\Hom{\operatorname{Hom}}

\def\Im{\operatorname{Im}}

\def\Int{\operatorname{Int}}

\def\Lie{\operatorname{Lie}}

\def\Pin{\operatorname{Pin}}
\def\rank{\operatorname{rank}}

\def\span{\operatorname{span}}
\def\Spin{\operatorname{Spin}}
\def\Stab{\operatorname{Stab}}
\def\tr{\operatorname{tr}}

\newcommand{\fra}{\mathfrak{a}}

\newcommand{\fre}{\mathfrak{e}}
\newcommand{\frf}{\mathfrak{f}}
\newcommand{\frg}{\mathfrak{g}}
\newcommand{\frh}{\mathfrak{h}}

\newcommand{\frk}{\mathfrak{k}}
\newcommand{\frl}{\mathfrak{l}}

\newcommand{\fru}{\mathfrak{u}}

\newcommand{\bbC}{\mathbb{C}}

\newcommand{\bbF}{\mathbb{F}}

\newcommand{\bbR}{\mathbb{R}}
\newcommand{\bbZ}{\mathbb{Z}}

\begin{document}

\title[Simple Lie groups of type $\bf E$]{Maximal abelian subgroups of compact simple Lie groups of
type $\bf E$}
\author{Jun Yu}
\date{}

\abstract{We classify closed abelian subgroups of a compact simple Lie group of adjoint type and of type
$\bf E_6$, $\bf E_7$ or $\bf E_8$ having centralizer of the same dimension as the dimension of the subgroup
and describe Weyl groups of maximal abelian subgroups.}
\endabstract

\maketitle



\tableofcontents

\section{Introduction}

In this paper we continue the study of maximal abelian subgroups of compact simple Lie groups in \cite{Yu2}
and \cite{Yu3}, where we give a classification of such subgroups of matrix groups, simple Lie groups of
type $\bf G_2$, $\bf F_4$ and $\bf D_4$, and finite maximal abelian subgroups of some spin and half-spin
groups. In this paper we discuss such subgroups of simple Lie groups of type $\bf E_6$, $\bf E_7$ and
$\bf E_8$. The same as in \cite{Yu2} and \cite{Yu3}, we study closed abelian subgroups $F$ of a compact
simple Lie group $G$ satisfying the condition of \[\dim\frg_0^{F}=\dim F,\tag{$\ast$}\] where $\frg_0=
\Lie G$ is the Lie algebra of $G$. Precisely, we discuss closed abelian subgroups of $G:=\Aut(\fru_0)$,
the automorphism group of a compact simple Lie algebra $\fru_0$ of type $\bf E_6$, $\bf E_7$ or $\bf E_8$.
In the $\bf E_6$ case, abelian subgroups contained in $G_0=\Int(\fre_6)$ and those being not contained in
$G_0$ are discussed separately and we employ slightly different methods to study them.

Given a closed abelian subgroup $F$ of $G=\Int(\fre_6)$, $\Aut(\fre_7)$, or $\Aut(\fre_8)$ satisfying the
condition $(\ast)$, by Lemma \ref{L:center}, we reduce this classification to the classification of finite
abelian subgroups of the derived groups of some Levi subgroups of $G$. By this in each case we can use
results in \cite{Yu2}, \cite{Yu3} and previous sections of this paper. The more difficult part is the
classification of finite abelian subgroups $F$ of $G$ satisfying the condition $(\ast)$. For this we recall
that a theorem of Steinberg indicates that the centralizer $G^{x}$ is connected if the order of an element
$x$ of $F$ is prime to the order of the center of the universal covering of $G$ and hence $G^{x}$ is semisimple
in this case. There are few semisimple subgroups of $G$ which are centralizers of elements and most of them
possess simple structure. The even more difficult is for finite 3-subgroups of $\Int(\fre_6)$ and finite
2-subgroups of $\Aut(\fre_7)$ or $\Aut(\fre_8)$. We show that there are actually two conjugacy classes of
finite 3-subgroups of $\Int(\fre_6)$ satisfying the condition $(\ast)$. The study of finite 2-subgroups
$F$ of $G=\Aut(\fre_7)$ or $\Aut(\fre_8)$ is more involved. After showing the exponent of $F$ is at most $4$,
we appeal to \cite{Huang-Yu} and \cite{Yu} for detailed knowledge of symmetric subgroups and elementary
abelian 2-subgroups of $G$. The classification reduces to the $\bf \E_6$ or matrix groups while $F$ contains
specific involutions or Klein four subgroups. We are able to show that there are few conjugacy classes of $F$
in the exceptional cases of $F$ containing no such involutions or Klein four subgroups. For closed abelian
subgroups of $G=\Aut(\fre_6)$ being not contained in $G_0$ and satisfying the condition $(\ast)$, by Lemma
\ref{L:center-outer}, we reduce the classification to the classification of finite abelian subgroups of
derived groups of some (disconnected) Levi subgroups. While this disconnected Levi subgroup is a proper
subgroup, its derived subgroup is related to the twisted unitray group or Spin group, in which case we could
apply results from \cite{Yu2} and \cite{Yu3}. For finite abelian subgroups, we show that there is a unique
conjugacy class of $F$ if it contains no outer involutions. While it contains an outer involution, the
question reduces to the group $\F_4$ or $\Sp(4)/\langle-I\rangle$, which are studied in \cite{Yu3} and
\cite{Yu2}.

In the process of our classification, the fusion of each abelian subgroup $F$ obtained is well understood. By
this we are able to describe its Weyl group $W(F):=N_{G}(F/C_{G}(F)$. We also discuss an example in the
$\bf E_8$ case where the small Weyl group of the associated fine group grading is a proper subgroup of the
Weyl group of the maximal abelian subgroup, which disproves a conjecture in \cite{Han-Vogan}. Our classification
in the $\bf E_6$ case gives a classification of fine group gradings of the complex simple Lie algebra
$\mathfrak{e}_{6}(\bbC)$ as a byproduct. After posing a manuscript containing this classification to
arXiv in 2012, the author was notified that a classification of fine group gradings of $\mathfrak{e}_{6}(\bbC)$
was independently obtained in \cite{Draper-Virtuel} in close time.







\section{Connected simple Lie group of type $\bf E_6$}

We follow the notation in \cite{Yu}, Subsection 3.1. Given a complex simple Lie algebra $\frg$, with a
Cartan subalgebra $\frh$, Killing form denoted by $B$, root system $\Delta=\Delta(\frg,\frh)$. For each
vector $\lambda\in\frh^{\ast}$, let $H_{\lambda}\in\frh$ be defined by \[B(H_{\lambda},H)=
\lambda(H),\ \forall H\in\frh,\] in particular we have a vector $H_{\alpha}\in\frh$ for each root $\alpha$.
Let \begin{equation} H'_{\alpha}=\frac{2}{\alpha(H_{\alpha})}H_{\alpha},\label{eq:coroot}\end{equation}
which is called a co-root vector. Choose a simple system $\Pi=\{\alpha_1,\alpha_2,\dots,\alpha_{r}\}$ of
the root system $\Delta$, with the Bourbaki numbering (\cite{Bourbaki}, Pages 265-300). For simplicity we
write $H_{i}$, $H'_{i}$ for $H_{\alpha_{i}}$, $H'_{\alpha_{i}}$, $1\leq i\leq r$. For a root $\alpha$,
denote by $X_{\alpha}$ a non-zero vector in the root space $\frg_{\alpha}$. One can normalize the root
vectors $\{X_{\alpha}\}$ appropriately so that the real subspace \begin{equation}\fru_0=
\span_{\bbR}\{X_{\alpha}-X_{-\alpha},i(X_{\alpha}+X_{-\alpha}),i H_{\alpha}:\alpha \in \Delta^{+}\}
\label{eq:compact real form}\end{equation} is a compact real form of $\frg$ (\cite{Knapp}, Pages 348-354).

Given a compact Lie group $G$, denote by $\frg_0$ the Lie algebra of $G$ and $\frg=\frg_0\otimes_{\bbR}\bbC$
the complexified Lie algebra of $G$. Denote by $\fre_6$ a compact simple Lie algebra of type $\bf E_6$ and by
$\E_6$ a connected and simply connected Lie group with Lie algebra $\fre_6$. Let $\fre_6(\bbC)$ and $\E_6(\bbC)$
denote their complexifications. Similar notations are used for other types. Given a compact semisimple Lie
algebra $\fru_0$, let $\Aut(\fru_0)$ be the group of automorphisms of $\fru_0$ and $\Int(\fru_0)=
\Aut(\fru_0)_0$ be the group of inner automorphisms.

\smallskip



The following proposition is a theorem of Steinberg (cf. \cite{Onishchik-Vinberg}, Pages 214-216) and
a well known result (cf. \cite{Knapp}, Page 260).

\begin{prop}\label{P:Steinberg-centralizer}
Given a connected and simply connected compact Lie group $G$ and an automorphism $\phi$ of $G$, the fixed point
subgroup $G^{\phi}$ is connected. Given a torus $S$ of a connected compact Lie group $G$, the centralizer
$C_{G}(S)$ is connected.
\end{prop}



We recall three lemmas from \cite{Yu3}.

\begin{lemma}\label{L:center}
If $F$ is a closed abelian subgroup of a connected compact Lie group $G$ satisfying the condition $(*)$,
then $\fra=Z(\frl)$ and $F':=F\cap L_{s}$ is a finite abelian subgroup of $L_{s}$ satisfying the condition
$(*)$, where $\fra=\Lie F\otimes_{\bbR}\bbC$, $\frl=C_{\frg}(\fra)$, $L=C_{G}(\fra)$, $L_{s}=[L,L]$.
\end{lemma}

\begin{lemma}\label{L:Levi}
Let $G$ be a compact connected and simply connected Lie group and $L$ be a closed connected subgroup of $G$.
If $\frl=\Lie L\otimes_{\bbR}\bbC$ is a Levi subalgebra of $\frg$, then $L_{s}$ is simply connected.
\end{lemma}

\begin{lemma}\label{L:Levi-center}
Let $G$ be a compact connected Lie group and $L$ be a closed connected subgroup of $G$. If $\frl=\Lie L
\otimes_{\bbR}\bbC$ is a Levi subalgebra of $\frg$, then $Z(L_{s})\subset Z(G)Z(L)_0$.
\end{lemma}



\smallskip

Let $G=\Aut(\fre_6)$. Let $G_1=\E_6$ and $G_2=\E_6\rtimes\langle\tau\rangle$, where $\tau^{2}=1$ and
$\fre_6^{\tau}=\frf_4$. Let $\pi: G_2\longrightarrow G$ be the adjoint homomorphism. Then, $\ker\pi=
\langle c\rangle$, where $c=\exp(\frac{2\pi i}{3}(H'_1+2H'_3+H'_5+2H'_6))$ with $o(c)=3$ and being a
generator of $Z(\E_6)$. Let \[\theta_1=\exp(\frac{2\pi i}{3}(2H'_1+H'_3)), \quad \theta_2=
\exp(\frac{2\pi i}{3}(2H'_1+2H'_3+2H'_4+H'_2+H'_5)).\] They are elements of order 3 in $G_1$. The root
system of $\frg^{\theta_1}$ is of type $3A_2$ with simple roots $\{\alpha_1,\alpha_3,\alpha_6,\alpha_5,
\beta-\alpha_2,\alpha_2\}$. The root system of $\frg^{\theta_2}$ is of type $D_4$ with simple roots
$\{\alpha_3,\alpha_4,\alpha_2,\alpha_5\}$, where $\beta=\alpha_1+\alpha_6+2(\alpha_2+\alpha_3+\alpha_5)+
3\alpha_4$ is the highest root. One has \[(G_1)^{\theta_1}\cong(\SU(3)\times\SU(3)\times\SU(3))/\langle
(\omega I,\omega I,\omega I)\rangle,\] where $c=(\omega I,\omega^{-1}I,I)$; and \[(G_1)^{\theta_2}\cong
(\Spin(8)\times\U(1)\times\U(1))/\langle(c_{8},-1,1),(-c_{8},1,-1)\rangle,\] where the Lie algebra of
$\U(1)\times\U(1)$ is $\span_{\bbR}\{2H'_1+2H'_3+2H'_4+H'_2+H'_5,H'_2+H'_3+2H'_4+2H'_5+2H'_6\}$,
$c=(1,\omega,\omega^{-1})$. Let $\eta\in\Aut(\frg)$ be an automorphism determined by
\[\eta(X_{\pm{\alpha_4}})=X_{\pm{\alpha_4}},\ \eta(X_{\pm{\alpha_2}})=X_{\pm{\alpha_3}},
\ \eta(X_{\pm{\alpha_3}})=X_{\pm{\alpha_5}},\ \eta(X_{\pm{\alpha_5}})=X_{\pm{\alpha_2}},\]
\[\eta(X_{\pm{\alpha_1}})=X_{\pm{\alpha_6}},\ \eta(X_{\pm{\alpha_6}})=X_{\mp{\beta}},
\ \eta(X_{\mp{\beta}})=X_{\pm{\alpha_1}}.\] Then, $\eta$ stabilizes $\fru_0$ and hence $\eta\in
\Aut(\fre_6)=G$. Since $\eta^3=1$, $\eta\in G_0$ and hence it lifts to an element of $G_1=\E_6$. In this
way we could regard $\eta$ as an element of $G$ or as an element of $G_1$. By Lemma \ref{L:E6-eta} below,
we see that $\eta$ is actually conjugate to $\theta_2$ in $G_1$. One has $\eta\theta_1\eta^{-1}=c\theta_1$
and $\eta\theta_2\eta^{-1}=c\theta_2$. For $i=1$, $2$, one has \[(G_0)^{\pi(\theta_{i})}=
\{g\in G_1: g\theta_{i}g^{-1}(\theta_{i})^{-1}\in\langle c\rangle\}/\langle c\rangle.\] Since
$\eta\theta_1\eta^{-1}=c\theta_1$, one has $\{g\in G_1: g\theta_{i}g^{-1}\theta_{i}^{-1}\in\langle c
\rangle\}=(G_1)^{\theta_1}\cdot\langle\eta\rangle$. Thus $(G_0)^{\pi(\theta_1)}$ has three connected
components. Furthermore one has \[G^{\pi(\theta_1)}\cong((\SU(3)\times\SU(3)\times\SU(3))/\langle
(\omega I,\omega I,\omega I),(\omega I,\omega^{-1} I,I)\rangle)\rtimes\langle\eta,\tau\rangle,\] where
$\eta(A_1,A_2,A_3)\eta^{-1}=(A_2,A_3,A_1)$ and $\tau(A_1,A_2,A_3)\tau^{-1}=(A_1,A_3,A_2)$. Similarly one has
\[G^{\pi(\theta_2)}\cong((\Spin(8)\times\U(1)\times\U(1))/\langle(c,-1,1),(-c,1,-1)\rangle)\rtimes\langle
\eta,\tau\rangle,\] where $(\mathfrak{so}(8)\oplus i\bbR\oplus i\bbR)^{\eta}=\frg_2\oplus 0\oplus 0$ and
$(\mathfrak{so}(8)\oplus i\bbR\oplus i\bbR)^{\tau}=\mathfrak{so}(7)\oplus\Delta(i\bbR)$.
Recall that $G$ has an involution $\sigma_1$ with (cf. \cite{Huang-Yu}, Table 2)
$$(G_0)^{\sigma_1}\cong(\SU(6)\times\Sp(1))/\langle(-I,-1),(\omega I,1)\rangle.$$


\smallskip

Given a closed abelian subgroup $F$ of $G$, for any $x,y\in F\cap G_0$, choosing $x',y'\in G_1$ such that
$\pi(x')=x$ and $\pi(y')=y$, then $[x',y']=c^{k}$ for some $k=0,1,2$. Define $m(x,y)=c^{k}$. Then, $m$ is
an antisymmetric bimultiplicative function on $F\cap G_0$. That means: for any $x_1,x_2,x_3\in F\cap G_0$,
$m(x_1,x_1)=1$, $m(x_1,x_2)=m(x_2,x_1)^{-1}$ and $m(x_1x_2,x_3)=m(x_1,x_3)m(x_2,x_3)$.

\begin{lemma}\label{L:E6-inner finite-order}
Let $F$ be a finite abelian subgroup of $G_0$ satisfying the condition $(*)$. For any $1\neq x\in F$,
\begin{enumerate}
\item {if $o(x)=p$ is a prime, then $p=2$ or $3$.}
\item {If $o(x)=2^{k}$, then $o(x)=2$ and $x\sim\sigma_1$.}
\item {If $o(x)=3^{k}$, then $o(x)=3$ and $x\sim\pi(\theta_1)$ or $\pi(\theta_2)$.}
\end{enumerate}
\end{lemma}

\begin{proof}
For $(1)$, suppose that $p>3$. Then there exists $y\in G_1$ such that $x=\pi(y)$ and $o(y)=o(x)=p$.
Thus $F\subset(G_0)^{x}=\{g\in G_1: gyg^{-1}y^{-1}\in\langle c\rangle\}/\langle c\rangle$. Since
$o(yc)=o(yc^{2})=3o(y)\neq o(y)$, one has $yc\not\sim y$ and $yc^{2}\not\sim y$. Hence $(G_0)^{x}=
(G_1)^{y}/\langle c\rangle$. By Steinberg's theorem $(G_1)^{y}$ and hence $(G_0)^{x}$ is connected.
Therefore $(G_0)^{x}$ is semisimple as $F$ satisfies the condition $(*)$. By this the root system of
$\frg^{y}$ is of type $3A_2$ or $A_5+A_1$ (cf. \cite{Oshima}). In this case $(G_0)^{x}$ does not contain
any element in its center of order a prime $p>3$, which gives a contradiction. For $(2)$, similar argument
as the above for $(1)$ shows that $(G_0)^{x}$ is connected and the root system of $\frg^{x}$ is of type
$3A_2$ or $A_5+A_1$. As it is assumed that $o(x)=2^{k}$, the root system of $\frg^{x}$ is of type $A_5+A_1$
and hence \[(G_0)^{x}\cong(\SU(6)\times\Sp(1))/\langle(\omega I,1),(-I,-1)\rangle.\] By this $o(x)=2$ and
$x\sim\sigma_1$. For $(3)$, suppose that $o(x)\neq 3$. Then $F$ has an element of order 9. We may and do
assume that $o(x)=9$. Choosing an element $y\in G_1$ such that $\pi(y)=x$, then $F\subset(G_1)^{y^3}/\langle c
\rangle$ since $x^3\in\ker m$. By Steinberg's theorem $(G_1)^{y^3}$ is connected. Since $F$ satisfies the
condition $(*)$, $(G_1)^{y^3}$ is semisimple. Therefore the root system of $\frg^{y}$ is of type $3A_2$ or
$A_5+A_1$. As it is assumed that $o(x)=9$, the root system of $\frg^{y^3}$ is of type $3A_2$. In this case
\[(G_1)^{y^3}/\langle c\rangle\cong(\SU(3)\times\SU(3)\times\SU(3))/\langle(\omega I,\omega I,\omega I),
(\omega I,\omega^{-1}I,I)\rangle,\] where $x^{3}=[y^3]=[(\omega I,I,I)]$ or $[(\omega^{-1}I,I,I)]$. Without
loss of generality, we may assume that $y=(A_1,A_2,A_3)$ with $(A_1)^{3}=\omega^{k}I$ and $(A_2)^{3}=(A_3)^{3}
=\omega^{k+1}I$ for some $k=0$, $1$ or $2$. If $k=1$ or $2$, then $A_1\not\sim\omega A_1$ and $A_1\not\sim
\omega^{2}A_1$. Hence the component of $((G_1)^{y^3}/\langle c\rangle)^{x}$ on the first $\SU(3)$ factor
is $\SU(3)^{A_1}$, which has a center of positive dimension. Thus $\frg_0^{F}\neq 0$. If $k=0$, then
$A_2\not\sim\omega A_2$, $A_2\not\sim\omega^{2}A_2$, $A_3\not\sim\omega A_3$ and $A_3\not\sim\omega^{2}A_3$.
In this case $(G_0)^{x}$ is connected and has a center of positive dimension. Again one has $\frg_0^{F}\neq 0$.
Thus $o(x)=3$. For the terminology in \cite{Andersen-Grodal-Moller-Viruel}, $G_0$ has five conjugacy classes
of elements of order 3 labeled as ``3B'',``3C'', ``3D'',``3E'',``3F''. For an element $x$ in the class ``3B'',
``3E'' or ``3F'', the order of $\Aut(\fre_6^{x})/\Int(\fre_6^{x})$ is coprime to 3 and $\fre_6^{x}$ is not
semisimple, hence $(G_0)^{x}$ is connected and is not semisimple. Therefore it does not contain any finite
abelian subgroup satisfying the condition $(*)$. Then, $x$ is in the class ``3C'' or ``3D''. That is,
$x\sim\pi(\theta_1)$ or $\pi(\theta_2)$.
\end{proof}

\begin{lemma}\label{L:E6-eta}
We have $\eta\sim\pi(\theta_2)$ and $\eta\pi(\theta_1)\sim\eta\pi(\theta_1)^{2}\sim\pi(\theta_1)$.
\end{lemma}

\begin{proof}
Let $A=\diag\{1,\omega,\omega^{2}\}$ and $B=\left(\begin{array}{ccc}0&1&0\\0&0&1\\1&0&0\\\end{array}
\right)$. Then \[F:=\langle\pi(\theta_1),\eta,(A,A,A),(B,B,B)\rangle\] is an elementary abelian 3-subgroup
of $G_0$ such that $\frg_0^{F}=0$. Hence any non-identity element of $F$ is conjugate to $\pi(\theta_1)$ or
$\pi(\theta_2)$ by Lemma \ref{L:E6-inner finite-order}. The Lie algebra $\frg=\fre_6(\bbC)$ has $36$ positive
roots, among them only the six roots \[\pm{}\alpha_4,\ \pm{}(\alpha_4+\alpha_2+\alpha_3+\alpha_5),
\ \pm{}(2\alpha_4+\alpha_2+\alpha_3+\alpha_5)\] are fixed under the action of $\eta$, and $\eta$ acts as
identity on their root vectors. Moreover the action of $\eta$ permutes the elements
$\{iH_{\alpha_{j}}: 1\leq j\leq 6, j\neq 4\}\cup\{iH_{-\beta}\}$. Hence \[\dim\frg^{\eta}=6+\frac{78-6}{3}=
30=\dim\frg^{\pi(\theta_2)}.\] Similar consideration shows \[\frg^{\eta\pi(\theta_1)}=
\frg^{\eta\pi(\theta_1)^{2}}=\frac{78-6}{3}=24=\dim\frg^{\pi(\theta_1)}.\] By these one gets
$\eta\sim\pi(\theta_2)$ and $\eta_1\pi(\theta_1)\sim\eta_1\pi(\theta_1)^{2}\sim\pi(\theta_1)$.
\end{proof}

Let $K_1=\langle\pi(\theta_1),\eta\rangle$, $K_2=\langle\pi(\theta_2),\eta\rangle$.

\begin{lemma}\label{L:E6-3 group-rank2}
Any rank 2 elementary abelian 3-subgroups of $G_0$ with a nontrivial function $m$ is conjugate to $K_1$
or $K_2$.
\end{lemma}

\begin{proof}
Let $F$ be a rank 2 elementary abelian 3-subgroup of $G_0$ with a nontrivial function $m$. By the same
reason as in the proof for Lemma \ref{L:E6-inner finite-order}, any non-identity element of $F$ is
conjugate to $\pi(\theta_1)$ or $\pi(\theta_2)$. If $F$ contains an element conjugate to $\pi(\theta_1)$,
without loss of generality we assume that $\pi(\theta_1)\in F$. Then, \[F\subset(G_0)^{\pi(\theta_{1})}=
((\SU(3)\times\SU(3)\times\SU(3))/\langle(\omega I,\omega I,\omega I),(\omega I,\omega^{-1} I,I)\rangle)
\rtimes\langle\eta\rangle.\] As the function $m$ on $F$ is nontrivial, $F$ is generated by $\pi(\theta_1)$
and another element of order 3 in the coset $\eta((G_0)^{\pi(\theta_{1})})_0$. Each element of
$\eta((G_0)^{\pi(\theta_{1})})_0$ is conjugate to an element of the form $(A,I,I)\eta$. Since
$((A,I,I)\eta)^{3}=1$ implies that $A\in\langle\omega I\rangle$, one has $F\sim\langle\pi(\theta_1),\eta
\rangle=K_1$ in this case. If $F$ contains no elements conjugate to $\pi(\theta_1)$, all non-identity elements
of $F$ are conjugate to $\pi(\theta_2)$. Without loss of generality we assume that $\pi(\theta_2)\in F$. Then,
\[F\subset(G_0)^{\pi(\theta_2)}=((\Spin(8)\times\U(1)\times\U(1))/\langle(c,-1,1),(-c,1,-1)\rangle)\rtimes
\langle\eta\rangle.\] As the function $m$ is nontrivial on $F$, it is generated by $\pi(\theta_2)$ and another
element of order 3 in the coset $\eta((G_0)^{\pi(\theta_{2})})_0$. Such an element is conjugate to an element
of the form $[(A,1,1)]\eta$, $A\in\Spin(8)$. Moreover it is conjugate to $\eta$ or $\eta'=[A]\eta$ (cf.
\cite{Helgason}), where $A\in\Spin(8)^{\eta}\cong\G_2$, $A^3=1$, $(\G_2)^{A}$ being a torus and
$(\Spin(8))^{\eta'}\cong\PSU(3)$. All nontrivial elements of $\langle\pi(\theta_2),\eta\rangle$ are conjugate
to $\pi(\theta_2)$, hence $\langle\pi(\theta_2),\eta\rangle\not\sim\langle\pi(\theta_1),\eta\rangle$. On the
other hand, one has $\langle\pi(\theta_2),\eta'\rangle\sim\langle\pi(\theta_1),\eta\rangle$ since the latter
contains an element conjugate to $\pi(\theta_2)$.
\end{proof}

\begin{lemma}\label{L:E6-eta'}
We have $\eta'\sim\eta'\pi(\theta_2)\sim\eta'\pi(\theta_2)^2\sim\pi(\theta_1)$.
\end{lemma}

\begin{proof}
By Lemma \ref{L:E6-eta}, the only elements conjugate to $\pi(\theta_2)$ in $\langle\pi(\theta_1),\eta\rangle$
are $\eta$ and $\eta^2$. Since $\langle\pi(\theta_2),\eta'\rangle\sim\langle\pi(\theta_1),\eta\rangle$ as
the above proof of Lemma \ref{L:E6-3 group-rank2} shows, one gets $\eta'\sim\eta'\pi(\theta_2)\sim
\eta'\pi(\theta_2)^2\sim\pi(\theta_1)$.
\end{proof}

In $(G_0)^{\sigma_1}\cong(\SU(6)\times\Sp(1))/\langle(-I,-1),(\omega I,1)\rangle$, let \[F_1=\langle[(-I,1)],
[(A_6,\mathbf{i})],[(B_6,\mathbf{j})]\rangle,\] where \[A_6=e^{\frac{2\pi i}{12}}\diag\{1,\omega_6,\omega_6^{2},
\omega_6^{3},\omega_6^{4},\omega_6^{5}\}\] and \[B_6=e^{\frac{2\pi i}{12}}\left(\begin{array}{cccccc}
0&1&&&&\\&0&1&&&\\&&0&1&&\\&&&0&1&\\&&&&0&1\\1&&&&&0\\\end{array}\right).\]

\begin{prop}\label{P:E6-inner finite 1}
Let $F$ be a finite abelian subgroup of $G_0$ satisfying the condition $(*)$. If $F$ is not a 3-group, then
$F\sim F_1$.
\end{prop}

\begin{proof}
By Lemma \ref{L:E6-inner finite-order}, we may assume that $\sigma_1\in F$. Then, \[F\subset(G_0)^{\sigma_1}
\cong(\SU(6)\times\Sp(1))/\langle(-I,-1),(\omega I,1)\rangle.\] For any element $(\lambda_1 I,\lambda_2)\in
\langle(-I,-1),(\omega I,1)\rangle$, one has $\lambda_2=\lambda_1^{3}$. Hence $F\sim F_1$ by \cite{Yu2},
Proposition 2.1.
\end{proof}

\begin{prop}\label{P:Weyl-E6 inner finite 1}
We have $W(F_1)=N_{G}(F_1)/C_{G}(F_1)\cong\GL(3,\mathbb{F}_2)\times\GL(2,\mathbb{F}_3)$.
\end{prop}

\begin{proof}
Let $K$ be the Sylow 3-subgroup of $F_1$. Then the function $m$ on $K$ is nontrivial and non-identity
elements of $K$ are conjugate to each other. Hence $K\sim K_2$ by Lemma \ref{L:E6-3 group-rank2}. Let
$K'$ be the Sylow 2-subgroup of $F_1$. There is an injecrtive homomorphism $\phi: W(F_1)
\longrightarrow W(K)\times W(K')$. By \cite{Yu} Proposition 6.10, $W(K')\cong\GL(3,\mathbb{F}_2)$; and
by the descripotion of $K_2$, $W(K_2)\cong\GL(2,\mathbb{F}_3)$. Since $$G^{\sigma_1}\cong((\SU(6)\times
\Sp(1))/\langle(-I,-1),(\omega I,1)\rangle)\rtimes\langle\tau\rangle,$$ where $$\tau[(X,\lambda)]\tau^{-1}
=[(J_3\overline{X}J_3^{-1},\lambda)],$$ $X\in\SU(6)$, $\lambda\in\Sp(1)$, we have $\Stab_{W(F_1)}(\sigma_1)
\cong P(1,2,\mathbb{F}_2)\times\GL(2,\mathbb{F}_3)$. Thus $|W(F_1)|=7\times |\Stab_{W(F_1)}(\sigma_1)|=
|\GL(3,\mathbb{F}_2)|\times|\GL(2,\mathbb{F}_3)|$. Therefore $W(F_1)\cong\GL(3,\mathbb{F}_2)\times
\GL(2,\mathbb{F}_3)$.
\end{proof}



In \[G^{\pi(\theta_1)}\cong((\SU(3)\times\SU(3)\times\SU(3))/\langle(\omega I,\omega I,\omega I),
(\omega I,\omega^{-1} I,I)\rangle)\rtimes\langle\eta\rangle,\] let \[F_2=\langle\pi(\theta_1),[(A,A,A)],
[(B,B,B)]\rangle\] and \[F_3=\langle\pi(\theta_1),\eta,[(A,A,A)],[(B,B,B)]\rangle.\] It is clear that
$F_3$ is a maximal abelian subgroup of $G$, while $F_2$ is not.

\begin{prop}\label{P:E6-inner finite}
Any elementary abelian 3-subgroup $F$ of $G_0$ satisfying the condition $(*)$ is conjugate to $F_2$ or $F_3$.
\end{prop}

\begin{proof}
If the function $m$ on $F$ is trivial, $F$ lifts to an isomorphic elementary abelian 3-subgroup $F'$ of
$G_1=\E_6$. By Lemma \ref{L:E6-inner finite-order}, any non-identity element of $F'$ is conjugate to
$\theta_1$. Without loss of generality we assume that $\theta_1\in F'$. Then \[F\subset(G_1)^{\theta_1}
\cong(\SU(3)\times\SU(3)\times\SU(3))/\langle(\omega I,\omega I,\omega I)\rangle\] and hence $F\sim F_2$.
If the function $m$ is nontrivial on $F$, by Lemma \ref{L:E6-3 group-rank2} $F$ contains a subgroup
conjugate to $K_1$ or $K_2$. We may assume that $K_1\subset F$ or $K_2\subset F$. In the first case, one
has \[F\subset(G_0)^{K_1}=\PSU(3)\times K_1\] and hence $F\sim F_3$. In the second case, one has
\[F\subset(G_0)^{K_2}=\G_2\times K_2.\] However $\G_2$ does not possess any elementary abelian 3-subgroup
satisfying the condition $(*)$.
\end{proof}

\begin{prop}\label{P:Weyl-E6-F2 and F3}
We have $W(F_2)\cong\SL(3,\mathbb{F}_3)$ and there is an exact sequence \[1\longrightarrow\mathbb{F}_3^{3}
\longrightarrow W(F_3)\longrightarrow P(2,1,\mathbb{F}_3)\longrightarrow 1.\]
\end{prop}

\begin{proof}
Since non-identity elements in $F_2$ are all conjugate to $\pi(\theta_1)$, $W(F_2)$ acts transitively on
$F_2-\{1\}$. On the other hand, from the construction of $F_2$ one can show that
$\Stab_{W(F_1)}(\pi(\theta_1))=\SL(2,\mathbb{F}_3)$. Therefore $W(F_2)\cong\SL(3,\mathbb{F}_3)$.
In $G^{\pi(\theta_1)}$, let \[F'_3=\langle\pi(\theta_1),[(A,A,A)],[(I,A,A^2)],[(B,B,B)]\rangle\] and
$K=\langle\pi(\theta_1),[(A,A,A)]\rangle$. Apparently $F'_3$ is a maximal finite abelian subgroup of $G_0$.
Hence $F'_3\sim F_3$. One can show that $$\ker(m|_{F'_3})=\langle\pi(\theta_1),[(A,A,A)]\rangle$$ and
\[\{x\in F'_3: x\in\pi(\theta_2)\}=\eta''\ker(m|_{F_3})\cup\eta''^2\ker(m|_{F_3}),\] where $\eta''=
[(I,A,A^2)]$. Let $$K=\langle\pi(\theta_1),[(A,A,A)],[(I,A,A^2)]\rangle,\quad K'=\langle\pi(\theta_1),
[(A,A,A)]\rangle.$$ Then, $K'$ and $K$ are stable under the action of $W(F_3)$. Hence there are homomorphisms
$p: W(F'_3)\longrightarrow W(K)$ and $p': W(F'_3)\longrightarrow W(K')$. By the uniqueness of conjugacy class
of maximal abelian subgroups of $G$ being isomorphic to $(c_3)^{4}$ and both $K.K'$ are stable under $W(F_3)$,
one sees that $p$ and $p'$ are surjective. The action of $W(K)$ also stabilizes $K'$, hence
$W(K)\subset P(2,1,\mathbb{F}_3)$. It is clear that $W(K')\cong\GL(2,\mathbb{F}_3)$. Calculation in
$G^{\pi(\theta_1)}$ shows \[C_{G}(K')=\langle F'_3,T_3\times T_3\times T_3,[(I,B,B^2)],\eta,\tau\rangle\]
and \[C_{G}(K)=\langle F'_3,T_3\times T_3\times T_3\rangle,\] where $T_3$ is the group of diagonal matrices
in $\SU(3)$. From the equation about $C_{G}(K')$, we get $\ker p'\cong\mathbb{F}_3^{3}\times
(\mathbb{F}_3^{3}\times\{\pm{1}\})$. Here we note that the elements $[(I,B,B^2)]$, $\eta$, $\tau\rangle$
commute with $[(B,B,B)]$ and their conjugation action could map $\eta''=[(I,A,A^{2})]$ to any element in
$\eta''\ker(m|_{F_3})\cup\eta''^2\ker(m|_{F_3})$; elements in $T_3\times T_3\times T_3$ commute with $\eta''$
and their conjugation action could map $[(B,B,B)]$ to any element in $[(B,B,B)]K$. Hence $|W(F_3)|=
|\GL(2,\mathbb{F}_3)|\times 2\times 3^2\times 3^3=|\mathbb{F}_3^{3}|\times|P(2,1,\mathbb{F}_3)|$. From the
equation about $C_{G}(K')$, we get $\ker p'\cong\mathbb{F}_3^{3}$ since the conjugation action of elements in
$T_3\times T_3\times T_3$ could map $[(B,B,B)]$ to any element in $[(B,B,B)]K$. In turn we get the exact
sequence in the proposition.
\end{proof}


Let \[H_1=(T_3\times\SU(3)\times\SU(3))/\langle(\omega I,\omega I,\omega I),(\omega I,\omega^{-1} I,I)\rangle,\]
and \[H_2=(\Spin(8)\times\U(1)\times\U(1))/\langle(c,-1,1),(-c,1,-1)\rangle,\] which are Levi subgroups of
$G_0$ with root systems $2A_2$ and $D_4$, where $T_3$ is a maximal torus of $\SU(3)$. In $H_1$, let \[F_4=
T_3\cdot\langle[(A,A)],[(B,B^{-1})]\rangle.\] In $H_2$, let \[F_5=\langle e_1e_2e_3e_4,e_5e_6e_7e_8,e_1e_2e_5e_6,
e_1e_3e_5e_7\rangle\cdot(\U(1)\times\U(1)).\]

\begin{prop}\label{P6:E6-inner-nonfinite}
Any closed non-finite abelian subgroup $F$ of $G_0$ satisfying the condition $(*)$ is either a maximal torus or
is conjugate to one of $F_4$, $F_5$.
\end{prop}

\begin{proof}
Let $L=C_{G_0}(\Lie F)$. By Lemma \ref{L:center}, $F=Z(L)_0\cdot(F\cap L_{s})$ and $F':=F\cap L_s$ is a
finite abelian subgroup of $L_{s}$ satisfying the condition $(*)$, where $L_{s}=[L,L]$. If $L_{s}$ is not
simply connected, then the root system of $\frl$ is of type $\emptyset$, $2A_2$ or $2A_2+A_1$. In the case
of the root system of $\frl$ is $\emptyset$, $F$ is a maximal torus of $G_0$. In the case of the root system
of $\frl$ is if type $2A_2$, $L_{s}\cong(\SU(3)\times\SU(3))/\langle(\omega I,\omega^{-1} I)\rangle$ and hence
$F\sim F_4$. In the case of the root system of $\frl$ is of type $2A_2+A_1$, \[L_{s}\cong(\SU(3)\times
\SU(3))/\langle(\omega I,\omega I)\rangle\times\Sp(1),\] which possesses no finite abelian subgroups satisfying
the condition $(*)$. If $L_{s}$ is simply connected, it can not have a type $A$ factor. Hence the root system of
$\frl$ is of type $D_4$ or $D_5$. By \cite{Yu3}, Proposition 2.2, in this case $F\sim F_5$.
\end{proof}

\begin{prop}\label{P:Weyl-E6-F4 and F5}
There are exact sequences \[1\longrightarrow(\mathbb{F}_3)^2\rtimes\SL(2,\mathbb{F}_3)\longrightarrow W(F_4)
\longrightarrow D_6\longrightarrow 1\] and \[1\longrightarrow\Hom((\mathbb{F}_2)^3,(\mathbb{F}_2)^2)\rtimes
\GL(3,\mathbb{F}_2)\longrightarrow W(F_5)\longrightarrow D_6\longrightarrow 1.\]
\end{prop}

\begin{proof}
Since $T_3$ is the maximal torus of a subgroup $H$ of $G$ isomorphic to $\SU(3)$, one has $W(T_3)\subset
\Aut(T)\cong D_6$. On the other hand, by embedding $H$ into $H_2$ one can show that $W(T_3)\cong D_6$.
There is a homomorphism $p: W(F_4)\longrightarrow W(T_3)$, which is clearly a surjective map. Using
\[C_{G}(T_3)=((T_3\times\SU(3)\times\SU(3))/\langle(\omega I,\omega I,\omega I)),(\omega I,\omega^{-1} I,I)
\rangle)\rtimes\langle\tau\rangle,\] where $\tau[(X_1,X_2,X_3)]\tau^{-1}=[(X_1,X_3,X_2)]$,
$X_1,X_2,X_3\in\SU(3)$, we get $\ker p\cong(\mathbb{F}_3)^2\rtimes\GL(2,\mathbb{F}_3)$. Therefore it follows
the exact sequence for $W(F_4)$.

Denote by $T'=(F_5)_0$. The integral weight lattice of $T'$ has a basis $\{\lambda_1,\lambda_2\}$, where
$\lambda_1=2\alpha_1+2\alpha_3+2\alpha_4+\alpha_2+\alpha_5$, $\lambda_2=\alpha_2+\alpha_3+2\alpha_4+2\alpha_5
+2\alpha_6$. Since $|\lambda_1|^2=|\lambda_2|^2=-2(\lambda_1,\lambda_2)$, the integral weight lattice of $T'$ is
isomorphic to the root lattice of $A_2$. Thus $\Aut(T')\cong D_6$. By emebedding $H_2$ into \[G^{\pi(\theta_2)}=
((\Spin(8)\times\U(1)\times\U(1))/\langle(c,-1,1),(-c,1,-1)\rangle)\rtimes\langle\eta,\tau\rangle,\] where
$(\mathfrak{so}(8)\oplus i\bbR\oplus i\bbR)^{\eta}=\frg_2\oplus 0\oplus 0$ and
$(\mathfrak{so}(8)\oplus i\bbR\oplus i\bbR)^{\eta}=\mathfrak{so}(7)\oplus\Delta(i\bbR)$, one gets $D_3\subset W(T')$.
On the other hand, the root system of type $D_4$, $\langle\alpha_2,\alpha_4,\alpha_3,\alpha_5\rangle$, is conjugate
to the root system $\langle\alpha_4,\alpha_2,\alpha_3+\alpha_4+\alpha_5,\alpha_1+\alpha_3+\alpha_4+\alpha_5+
\alpha_6\rangle$. As $\tau$ acts as identity on the latter, there is an outer involution acting as identity on
$\langle\alpha_2,\alpha_4,\alpha_3,\alpha_5\rangle$ and hence as $-1$ on $\{\lambda_1,\lambda_2\}$. That means
$-1\in W(T')$. Therfore $W(T')\cong D_6$. There is a homomorphism $p: W(F_5)\longrightarrow W(T')$, which is
apparently a surjective map. Using $\pi(\theta_2)\in T'$ and the expression of $G^{\pi(\theta_2)}$, one shows
$C_{G}(T')=H_2$. By this we get $\ker p\cong\Hom((\mathbb{F}_2)^3,(\mathbb{F}_2)^2)\rtimes\GL(3,\mathbb{F}_2)$.
\end{proof}



\section{Disconnected simple Lie group of type $\bf E_6$}

Following the notation of last section, we have $G=\Aut(\fre_6)$, $G_1=\E_6$, $G_2=\E_6\rtimes\langle\tau
\rangle$, where $\tau^{2}=1$ and $\fre_6^{\tau}=\frf_4$. Let $\pi: G_2\longrightarrow G$ be the adjoint
homomorphism with $\ker\pi=\langle c\rangle$, $c=\exp(\frac{2\pi i}{3}(H'_1+2H'_3+H'_5+2H'_6))$.

\begin{lemma}\label{L:E6-outer finite-order1}
If $F$ is a finite abelian subgroup of $G$ satisfying the condition $(*)$ and being not contained in $G_0$, then for
any element $y\in F-F\cap G_0$, $o(y)=2,4,6$ or $8$.
\end{lemma}

\begin{proof}
Since $F$ is an abelian subgroup of $G$ not contained in $G_0$, the function $m$ on $F\cap G_0$ is trivial.
By Steinberg's theorem $F\cap G_0\subset(G^{x})_0$ for any $x\in F\cap G_0$. For any $y\in F-F\cap G_0$, let
$x=y^2$. Then we have $F\subset\langle(G^{x})_0,y\rangle$. Let $Z_{x}=Z((G^{x})_0)_0$. Since $Z_{x}$ commutes
with $F\cap G_0$, $y^2=x\in F\cap G_0$ and $\frg_0^{F}=0$, $\Ad(y)$ acts as $-1$ on $Z_{x}$.
Apparently one has $G^{y}\subset G^{x}$. By \cite{Yu2} Lemma 5.3, one has $\rank G^{y}=4$. Hence
$\rank[\frg_0^{x},\frg_0^{x}]^{y}=4$, in particular $\rank[\frg_0^{x},\frg_0^{x}]\geq 4$. Write $x$ as the form
$x=x_{s}x_{z}$ with $x_s\in Z([(G^{x})_0,(G^{x})_0])$ and $x_{z}\in Z_{x}$. The root system of $\frg^{x}$ is one
of the following types (cf. \cite{Oshima}) $D_4$, $A_4$, $A_3+A_1$, $2A_2$, $A_2+2A_1$, $4A_1$, $D_5$, $A_5$,
$A_4+A_1$, $A_3+2A_1$, $2A_2+A_1$, $A_5+A_1$, $3A_2$. We may assume that $\Ad(y)$ maps a Cartan subalgebra $\frh$
of $\frg$ to itself and furthermore it maps a simple system $\Pi$ of the root system $\Delta=\Delta(\frg,\frh)$
to itself.

In the case of $\rank[\frg_0^{x},\frg_0^{x}]=4$, $y$ acts on $[\frg_0^{x},\frg_0^{x}]$ as an inner automorphism
since $\rank[\frg_0^{x},\frg_0^{x}]^{y}=4$. Hence $y$ commutes with $x_{s}$. Thus $y$ commutes with $x_{z}$ as it
commutes with $x$. Therefore $x_{z}^{2}=1$. If the root system of $\frg^{x}$ is of type $A_4$, without loss of
generality we may assume that it has $\{\alpha_1,\alpha_3,\alpha_4,\alpha_2\}$ as a simple system. The orthogonal
complement of it is $\span\{\alpha_6,\gamma\}$, where $\gamma=2\alpha_1+4\alpha_3+6\alpha_4+3\alpha_2+5\alpha_5$.
Thus $\Ad(y)$ acts as identity on $\span\{\alpha_1,\alpha_3,\alpha_4,\alpha_2\}$ and as $-1$ on
$\span\{\alpha_6,\gamma\}$. In this case $\Ad(y)\alpha_5\not\in\Delta$, which is a contradiction.
If the root system of $\frg^{x}$ is of type $A_3+A_1$, without loss of generality we may assume that
it has $\{\alpha_3,\alpha_4,\alpha_5\}\cup\{\beta\}$ as a simple system. The orthogonal complement of it is
$\span\{\gamma_1,\gamma_2\}$. Here $\beta=\alpha_1+2\alpha_2+2\alpha_3+3\alpha_4+2\alpha_5+\alpha_6$,
$\gamma_1=4\alpha_1+3\alpha_3+2\alpha_4+\alpha_5$, $\gamma_2=\alpha_3+2\alpha_4+3\alpha_5+4\alpha_6$. Thus $\Ad(y)$
acts as identity on $\span\{\alpha_3,\alpha_4,\alpha_5,\beta\}$ and as $-1$ on $\span\{\gamma_1,\gamma_2\}$.
In this case $\Ad(y)\alpha_1\not\in\Delta$, which is a contradiction. If the root system of $\frg^{x}$ is of type
$2A_2$, without loss of generality we may assume that it has $\{\alpha_1,\alpha_3\}\cup\{\alpha_5,\alpha_6\}$ as a
simple system. The orthogonal complement it is $\span\{\beta,\alpha_2\}$. Thus $\Ad(y)$ acts as identity on
$\span\{\alpha_1,\alpha_3,\alpha_5,\alpha_6\}$ and as $-1$ on $\span\{\beta,\alpha_2\}$. In this case
$\Ad(y)\alpha_4\not\in\Delta$, which is a contradiction. If the root system of $\frg^{x}$ is of type $A_2+2A_1$,
without loss of generality we may assume that it has $\{\alpha_2,\alpha_4\}\cup\{\alpha_1\}\cup\{\alpha_6\}$ as a
simple system. The orthogonal complement of it is $\span\{\gamma_4,\gamma_5\}$, where
$\gamma_4=3\alpha_1+6\alpha_3+4\alpha_4+2\alpha_2$, $\gamma_5=2\alpha_2+4\alpha_4+6\alpha_5+3\alpha_6$. Thus $\Ad(y)$
acts as identity on $\span\{\alpha_1,\alpha_2,\alpha_4,\alpha_6\}$ and as $-1$ on $\span\{\gamma_4,\gamma_5\}$.
In this case $\Ad(y)\alpha_3\not\in\Delta$, which is a contradiction. If the the root system of $\frg^{x}$ is of
type $D_4$ or $4A_1$, then $x_{s}^{2}=1$ and hence $x^2=(x_s)^2(x_z)^2=1$.

In the case of $\rank[\frg_0^{x},\frg_0^{x}]=5$, the root system of $\frg^{x}$ is not of type $A_5$ or $A_4+A_1$
since $\rank[\fru_0^{x},\fru_0^{x}]^{y}=4$ can not hold for these cases. If the root system $\frg^{x}$ is of type
$2A_2+A_1$, similarly as in the $A_4$ case we get a contraction by considering the action of $\Ad(y)$ on $\Delta$.
If the root system of $\frg^{x}$ is of type $D_5$, one has $(G^{x})_0=(\Spin(10)\times\U(1))/\langle(c,i)\rangle$.
Thus $x=(1,t)$ for some $t\in\U(1)$. Therefore $x=yxy^{-1}=x^{-1}$ and hence $x^{2}=1$. If the root system of
$\frg^{x}$ is of type $A_3+2A_1$, without loss of generality we may assume that it has
$\{-\beta,\alpha_2,\alpha_4\}\sqcup\{\alpha_1\}\sqcup\{\alpha_6\}$ as a simple system. Then,
\[(G^{x})_0\cong(\SU(4)\times\Sp(1)\times\Sp(1)\times\U(1))/\langle(iI,-1,1,i),(I,-1,-1,-1)\rangle.\]
Hence $x=[(1,1,(-1)^{k},t)]$ for some $t\in\U(1)$ and $k=0$ or $1$. If $\Ad(y)$ stabilizes the two roots
$(\alpha_1,\alpha_6)$, then $x=yxy^{-1}=(1,1,(-1)^{k},t^{-1})=x^{-1}$. Hence $x^{2}=1$. If $\Ad(y)$ permutes
the two roots $(\alpha_1,\alpha_6)$, then \[x=yxy^{-1}=(1,(-1)^{k},1,t^{-1})=x(1,1,1,(-1)^{k}t^{-2}).\] Thus
$t^{2}=(-1)^{k}$. In this case $x$ is a power of $x_0=(1,1,-1,i)$. As $o(x_0)=4$, we have $o(x)=1$, $2$ or $4$.

In the case of $\rank[\frg_0^{x},\frg_0^{x}]=6$, the root system of $\frg^{x}$ is of type $A_5+A_1$ or $3A_2$.
Thus $o(x)=2$ or $3$.
\end{proof}


For the involution $\sigma_1\in G$, we have \[G^{\sigma_1}=((\SU(6)\times\Sp(1))/\langle(\omega I,1),
(-I,-1)\rangle)\rtimes\langle\tau\rangle,\] where $\tau^2=1$ and $\tau([X,\lambda])\tau^{-1}=
[(J_3\overline{X}(J_3)^{-1},\lambda)]$, $X\in\SU(6)$, $\lambda\in\Sp(1)$. Let \[F_6=\langle[(J_3,1)\tau],
[(\diag\{I_{1,1},I_{1,1},J_1\},\textbf{i})],[(\diag\{J'_{1},J_1,I_{1,1}\},\textbf{j})]\rangle.\] It is
clear that $F_6$ is a maximal abelian subgroup of $G$.

\begin{lemma}\label{L:E6-outer finite1}
Any finite abelian subgroup $F$ of $G$ satisfying the condition $(*)$ and being not contained in $G_0$
either contains an outer involution, or is conjugate to $F_6$.
\end{lemma}

\begin{proof}
Suppose that $F$ does not contain an outer involution. By Lemma \ref{L:E6-outer finite-order1},
$o(y)=4$ or $8$ for any $y\in F-F\cap G_0$. In the case of $o(y)=4$, let $x=y^2$. Then $x\sim\sigma_1$
or $\sigma_2$. If $x\sim\sigma_1$, we may and do assume that $x=\sigma_1$. Then we have
$$F\subset G^{\sigma_1}=((\SU(6)\times\Sp(1))/\langle(\omega I,1),(-I,-1)\rangle)\rtimes\langle\tau
\rangle.$$ Since $y^2=x=\sigma_1$, one has $y\sim[(I,\mathbf{i})]\tau$ or $[(J_3,1)]\tau$. Then,
$F\subset((\Sp(3)\times\U(1))/\langle(-I,-1)\rangle)\cdot \langle y\rangle$ or $F\subset((\SO(6)
\times\Sp(1))/\langle(-I,-1)\rangle)\cdot \langle y\rangle$. Since the first group has a one-dimensional
center, it does not possess any finite abelian subgroup satisfying the condition $(*)$. For the second
group, let $F'$ be the image of the projection of $F$ to $\SO(6)/\langle-I\rangle$. In the terminology
of \cite{Yu2} Section 3, one has $k=1$ and $s_0=3$. Moreover considering determinant one gets
$\mu_1\neq\mu_2,\mu_3$ and $\mu_2\neq\mu_3$. Therefore $F\sim F_6$. If $x\sim\sigma_2$, we may and do
assume that $x=\sigma_2$. Then, \[F\subset G^{\sigma_2}=((\Spin(10)\times\U(1))/\langle(c,i)\rangle)
\rtimes\langle\tau\rangle,\] where $\tau([x,\lambda])=[(e_{10}x e_{10}^{-1},\lambda^{-1})]$,
$x\in\Spin(10)$, $\lambda\in\U(1)$. Since $y^2=x=\sigma_1$, one has $y\sim(e_1e_2,1)\tau$ or
$y\sim(e_1e_2e_3e_4e_5e_6,1)\tau$. Thus \[F\subset G^{y}=((\Spin(7)\times\Spin(3))/\langle(-1,-1)\rangle)
\cdot\langle y\rangle.\] For the reason of clarity, we denote by $e_1,e_2,\dots,e_7$ and $e_8,e_9,e_{10}$
respectively orthonomal base of Euclidean spaces defining $\Spin(7)$ and $\Spin(3)$. Let $F'$ be the
intersection of $F$ with $(\Spin(7)\times\Spin(3))/\langle(-1,-1)\rangle$. We may assume that
$F'\subset(\langle e_1e_2,\dots,e_1e_7\rangle\times\langle e_8e_9,e_8e_{10}\rangle)/\langle-1\rangle$.
Let $F'_1$, $F'_2$ be the images of the projections of $F'$ to $\Spin(7)/\langle-1\rangle$ and
$\Spin(3)/\langle-1\rangle$, and $F''_1=F\cap\Spin(7)$. Since $F$ does not contain any outer involution,
for any $x\in F''_1$, $x^2=1$; for any $x=[(x_1,x_2)]\in F'$ with $x_2\neq\pm{1}$, $x_1^{2}=-1$. By this
one can show that $\mathfrak{so}(7)^{F'_1}$ can not hold. This is a contradiction. In the case of $o(y)=8$,
let $x=y^{2}$. By the proof of Lemma \ref{L:E6-outer finite-order1}, one has \[(G_0)^{x}\cong(\SU(4)\times
\Sp(1)\times\Sp(1)\times\U(1))/\langle(iI,-1,1,i),(I,-1,-1,-1)\rangle,\] the action of $\Ad(y)$ on $\frg^{x}$
permutes the two $\mathfrak{sp}(1)$ factors, and $x=[(1,1,-1,i)]$. One can show that \[G^{x}\cong((\SU(4)
\times\Sp(1)\times\Sp(1)\times\U(1))/\langle(iI,-1,1,i),(I,-1,-1,-1)\rangle)\rtimes\langle\tau\rangle,\]
where $\tau[(X,A,B,\lambda)]\tau^{-1}=[(X,B,A,\lambda^{-1})]$, $X\in\SU(4)$, $A,B\in\Sp(1)$,
$\lambda\in\U(1)$. Let $y=[(X,A,B,\lambda)]\tau$. Then \[y^{2}=[(X^{2},AB,BA,1)].\] Since $y^2=x=[(1,1,-1,i)]
=[(-iI,-1,-1,1)]=[(iI,1,1,1)]$, one has $y\sim(\pm\frac{1\pm{i}}{\sqrt{2}}I_{1,3},\pm{1},1,1)\tau$.
Then, $\dim Z(G^{y})=1$ and hence $F$ does not satisfy the condition $(*)$.
\end{proof}


\begin{prop}\label{P:Weyl-E6-F6}
There is an exact sequence \[1\longrightarrow\Hom'((\mathbb{F}_2)^3,(\mathbb{F}_2)^3)
\longrightarrow W(F_6)\longrightarrow P(2,1,\mathbb{F}_2)\longrightarrow 1,\] where
$\Hom'((\mathbb{F}_2)^3,(\mathbb{F}_2)^3)$ means the space of linear maps
$f: (\mathbb{F}_2)^3\longrightarrow(\mathbb{F}_2)^3$ with $\tr f=0$
\end{prop}

\begin{proof}
Let $K=\{x\in F_6: x^2=1\}$ and $K'=\{x\in K: x\sim\sigma_2\}\cup\{1\}$ (cf. \cite{Huang-Yu}).
Then we have $K=\{x^2: x\in F_6\}$. Morever $K$ and $K'$ are elementary abelian 2-subgroups with
$\rank K=3$ and $\rank K'=2$. In the terminology of \cite{Yu}, $K$ is the subgroup $F'_{2,1}$ and one
has $W(K)\cong P(2,1,\mathbb{F}_2)$ by \cite{Yu}, Proposition 6.10. There is a homomorphism $p: W(F_6)
\longrightarrow W(K)$, which is clearly a surjective map. By considering in \[G^{\sigma_1}=
((\SU(6)\times\Sp(1))/\langle(\omega I,1),(-I,-1)\rangle)\rtimes\langle\tau\rangle,\] one shows
\[C_{G}(K)=\langle F_6, [(X_1,X_2,X_3,\lambda)]: X_1,X_2,X_3\in\U(2),\det(X_1X_2X_3)=1,\lambda\in
\Sp(1)\rangle.\] Using the elements $(iI_{1,1},1,1)$, $(1,iI_{1,1},1)$, $(1,1,iI_{1,1})$, $(J_1,1,1)$,
$(1,J_1,1)$, $(1,1,J_1)$, $(iI_2,-iI_2,I_2)$, $(I_2,iI_2,-iI_2)$ of $\U(2)^{3}$, one can show that
$\Hom'((\mathbb{F}_2)^3,(\mathbb{F}_2)^3)\subset\ker p$, where we identify both of $F_6/K$ and $K$ with
$(\mathbb{F}_2)^3$. On the other hand, if an element $g\in C_{G}(K)$ also cummutes with
$[(\diag\{I_{1,1},I_{1,1},J_1\},\textbf{i})]$ and $[(\diag\{J'_{1},J_1,I_{1,1}\},\textbf{j})]$, then it
lies in $\langle F_6, [(\lambda_1 I_2,\lambda_2 I_2,\lambda_3 I_3,\lambda)]: \lambda_1,\lambda_2,
\lambda_3\in\U(1),\lambda_1\lambda_2\lambda_3=1,\lambda\in\Sp(1)\rangle$. Thus
$g([(J_3,1)]\tau)g^{-1}([(J_3,1)]\tau)^{-1}$ can not be equal to $\sigma_1=[(I,-1)]$ and hence $\ker p$
is a proper subgroup of $\Hom((\mathbb{F}_2)^3,(\mathbb{F}_2)^3)$. Therfore $\ker p=\Hom'((\mathbb{F}_2)^3,
(\mathbb{F}_2)^3)$.
\end{proof}

In $G^{\sigma_3}=\F_4\times\langle\sigma_3\rangle$, let $F_7$ be a subgroup isomorphic to $(C_3)^{3}
\times C_2$, containing $\sigma_3$ an with each element $x$ of order $3$ satisfying that $\frf_4^{x}
\cong\mathfrak{su}(3)\oplus\mathfrak{su}(3)$ (\cite{Yu3}, Section 3); let $F_8$ be an elementary
abelian 2-subgroup of rank $6$ with $\sigma_3\in F_7$, i.e., the subgroup $F_{2,3}$ as in \cite{Yu},
Section 6;  In $G^{\sigma_4}=\Sp(4)/\langle-I\rangle\times\langle\sigma_3\rangle$, let
\[F_9=\langle[\textbf{i}I],[\textbf{j}I],[\diag\{-1,-1,1,1\}],[\diag\{-1,1,-1,1\}]\rangle],\]
\[F_{10}=\langle[\textbf{i}I],[\textbf{j}I],[I_{2,2}],[\diag\{I_{1,1},I_{1,1}\}],[\diag\{J'_{1},J'_{1}\}],\]
\[F_{11}=\langle[\textbf{i}I],[\textbf{j}I],[I_{2,2}],[\diag\{I_{1,1},I_{1,1}\}],[\diag\{J'_{1},J_{1}]\},\]
\[F_{12}=\langle[\textbf{i}I],[\textbf{j}I],[I_{2,2}],[J'_2],[\diag\{I_{1,1},I_{1,1}\}],[\diag\{J'_{1},J'_{1}]\}.\]
Among $F_7$, $F_{8}$, $F_{9}$, $F_{10}$, $F_{11}$, $F_{12}$, the maximal abelian ones are $F_7$, $F_{8}$, $F_{11}$,
$F_{12}$.

\begin{prop}\label{P:E6-outer finite2}
Let $F$ be a finite abelian subgroup of $G$ satisfying the condition $(*)$ and being not contained in
$G_0$. If $F$ contains an outer involution, then it is conjugate to one of $F_7$, $F_{8}$, $F_{9}$,
$F_{10}$, $F_{11}$, $F_{12}$.
\end{prop}

\begin{proof}
Since $F$ contains an outer involution, we may assume that it contains $\sigma_3$, or it contains
$\sigma_4$ and without elements conjugate to $\sigma_3$. In the case of $\sigma_3\in F$, by \cite{Yu3},
Proposition 3.2, one has $F\sim F_7$ or $F\sim F_8$. In the case of $F$ contains $\sigma_4$ and without
elements conjugate to $\sigma_3$, one has $F\subset G^{\sigma_4}=\Sp(4)/\langle-I\rangle\times\langle
\sigma_3\rangle$. Let $F'=F\cap\Sp(4)/\langle-I\rangle$. Since $F$ contains no elements conjugate to
$\sigma_3$, $F'$ contains no elements conjugate to $[I_{1,3}]$ in $\Sp(4)/\langle-I\rangle$
(cf. \cite{Huang-Yu}, Page 413). By \cite{Yu2}, Proposition 4.1, it is associated two integers $s_0\geq 1$
and $k\geq 1$ with $4=2^{k-1}s_0$. Then, $k=1$, $2$, $3$. If $k=1$, then $s_0=4$ and $F\sim F_9$. If
$k=2$, then $s_0=2$ and we have two functions $\mu_1,\mu_2$ on $F'/B_{F'}$. If $\mu_1=\mu_2$, then
$F\sim F_{10}$. If $\mu_1\neq\mu_2$, then $F\sim F_{11}$. If $k=3$, then $F\sim F_{12}$.
\end{proof}

\begin{prop}\label{P:Weyl-E6-F7 to F12}
We have \[W(F_7)\cong\SL(3,\mathbb{F}_3),\]
\[W(F_8)\cong(\mathbb{F}_2)^2\rtimes P(2,3;\mathbb{F}_2)),\]
\[W(F_9)\cong(\mathbb{F}_2)^4\rtimes((\Hom((\mathbb{F}_2)^2,(\mathbb{F}_2)^2)
\rtimes(\GL(2,\mathbb{F}_2)\times\GL(2,\mathbb{F}_2))),\]
\[W(F_{10})\cong(\mathbb{F}_2)^5\rtimes\Sp(1,1;0,1),\]
\[|W(F_{11})|=3\times 2^{12},\]
\[W(F_{12})\cong(\mathbb{F}_2)^6\rtimes\Sp(2;0,1),\]
\end{prop}

\begin{proof}
For $F_7$, $W(F_7)$ must fixes $\sigma_3$. The conclusion follows from \cite{Yu3}, Proposition 3.
Since $F_8$, $F_{9}$, $F_{10}$ and $F_{12}$ are elementary abelian 2-subgroups, the description of their
Weyl groups is given in \cite{Yu}, Proposition 6.10. For $F_{11}$, let $K=\{x\in F_{11}: x^2=1\}$,
$K'=K\cap G_0$ and $z=[I_{2,2}]$. Since $\{x^2: x\in F_{11}\}=\langle z\rangle$, the action of
$W(F_{11})$ on $F_{11}$ fixes $z$ and stabilizes $K$. Hence there is a homomorphism $p: W(F)
\longrightarrow\Stab_{W(K)}(z)$. One can calculate that $C_{G}(K)=\langle F_{11},[J'_2],
[\diag\{J'_1,J'_1\}]\rangle$. The conjugation action of $[J'_2]$, $[\diag\{J'_1,J'_1\}]$ maps
$[\diag\{J'_{1},J_{1}\}]$ to the element $[\diag\{I_{1,1},I_{1,1}\}]\diag\{J'_{1},J_{1}\}]$,
$[I_{2,2}\diag\{J'_{1},J_{1}\}]$ respectively. Hence $\ker p\cong(\mathbb{F}_2)^2$. We show that $p$
is surjective. Since the elements in $\sigma_4 K'$ are all conjugate to $\sigma_4$, both $W(F_{11})$
and $\Stab_{W(K)}(z)$ act transively on this set. Note that $\Stab_{W(K)}(z)$ is an index $3$
subgroup of $W(K)$ while $\sigma_4 K'$ has $16$ elements, hence $\Stab_{W(K)}(z)$ also acts transitively
on the set $\sigma_4 K'$. By this, to show $p$ is surjective we just need to show
$\Stab_{W(F_{11})}(\sigma_4)\longrightarrow\Stab_{W(K)}(\{z,\sigma_4\})$ is surjective. Consideration in
$G^{\sigma_4}$ shows it is of this case. By \cite{Yu}, Proposition 6.10, one has
$|W(K)|=16\times|\Sp(2,0;0,1)|=16\times 16\times 6^2=3^2\times 2^{10}$. Therefore
$|W(F_{11})|=\frac{4}{3}|W(K)|=3\times 2^{12}$.
\end{proof}

The proof of the following lemma is along the same line as that of \cite{Yu3}, Lemma 4.1, we omit it
here.
\begin{lemma}\label{L:center-outer}
Let $F$ be a non-finite closed abelian subgroup of $G$ satisfying the condition $(*)$ and being not
contained in $G_0$. Then $\Lie F\otimes_{\bbR}\bbC$ is conjugate to the center of a Levi subalgebra of
$\fre_6(\bbC)^{\sigma_3}$.
\end{lemma}

Let $H_1,H_2,H_3$ be the following subgroups of $G=\Aut(\fre_6)$,
\[H_1=((\SU(3)\times\SU(3)\times\SU(3))/\langle(\omega I,\omega I,\omega I),(I,\omega I,\omega^{-1}I)
\rangle)\rtimes\langle\tau\rangle,\]
\[H_2=((\SU(6)\times\Sp(1))/\langle(\omega I,1),(-I,-1)\rangle)\rtimes\langle\tau\rangle,\]
\[H_3=((\Spin(10)\times\U(1))/\langle(c,i)\rangle)\rtimes\langle\tau\rangle,\] where
$$\tau[(X_1,X_2,X_3)]\tau^{-1}=[(X_1,X_3,X_2)],$$
$$\tau[(X,\lambda)]\tau^{-1}=[(J_3\overline{X}J_3^{-1},\lambda)],$$
$$\tau[(x,\mu)]\tau^{-1}=[(e_{10}xe_{10}^{-1},\mu^{-1})],$$ for $X_1,X_2,X_3\in\SU(3)$, $X\in\SU(6)$,
$x\in\Spin(10)$, $\lambda\in\Sp(1)$, $\mu\in\U(1)$. The root system of the subalgebra $\frh_1^{\tau}$,
$\frh_2^{\tau}$, $\frh_3^{\tau}$ is the sub-root system $A_2^{L}+A_2^{S}$, $A_1^{L}+C_3$, $B_4$ of
$\F_4$, the root system of $\fre_6(\bbC)^{\tau}$.

Denote by $\frk=\fre_6(\bbC)^{\tau}$. By Lemma \ref{L:center-outer}, we may assume that $\fra\subset
\frk$, where $\fra=\Lie F\otimes_{\bbR}\bbC$. Write for $K=G^{\sigma_3}$, $\frl=C_{\frg}(\fra)$,
$\frl'=C_{\frk}(\fra)$, $L=C_{G}(\fra)$, $L'=C_{K}(\fra)$, $A=Z(L)_0$.


\begin{lemma}\label{L:E6-outer-nonfinite2}
Let $F$ be a non-finite closed abelian subgroup of $G$ satisfying the condition $(*)$ and being not
contained in $G_0$. If $\fra\subset\frk$, then the root system of the Levi subalgebra
$\frl'=C_{\frk}(\fra)$ is one of the types $\emptyset$, $B_2$, $B_3$, $C_3$.
\end{lemma}

\begin{proof}
As $\frl'$ is a Levi subalgebra of $\frk=\frf_4(\bbC)$, its root system is one of the following types:
$\emptyset$, $A_1^{L}$, $A_1^{S}$, $A_2^{L}$, $A_2^{S}$, $A_1^{L}+A_1^{S}$, $B_2$, $B_3$, $C_3$,
$A_2^{L}+A_1^{S}$, $A_2^{S}+A_1^{L}$. Each of them is contained in one of $B_4$, $A_2^{L}+A_2^{S}$,
$A_1^{L}+C_3$. Hence a conjugate of $L$ is contained in some $H_{i}$, $i=1$, $2$ or $3$. In the case
of the root system of $\frl'$ is of type $A_1^{L}$, one has $L\subset H_1$ and \[L=(\U(2)\times T_3
\times T_3)/\langle(\omega I,\omega I,\omega I),(I,\omega I,\omega^{-1}I)\rangle\rtimes\langle\tau
\rangle.\] Let $L''=(\SU(2)\times(T'_3/\langle\omega I\rangle))\rtimes\langle\tau\rangle$, where
$T_3'=\{(X,X^{-1}):X\in T_3\}\subset T_3\times T_3$. Then $F'':=F\cap L''$ is a finite abelian subgroup
of $L''$ satisfying the condition $(*)$. Since $\SU(2)$ is a direct factor of $L''$, it does not possess
any finite abelian subgroup satisfying the condition $(*)$, which is a contradiction. In the case of
the root system of $\frl'$ is of type $A_1^{S}$, one has $L\subset H_1$ and \[L=(T_3\times\U(2)\times
\U(2))/\langle(\omega I,\omega I,\omega I),(I,\omega I,\omega^{-1}I)\rangle\rtimes\langle\tau\rangle.\]
Let $L''=((\SU(2)\times\SU(2)\times\U(1))/\langle(-I,-I,-1),(I,I,\omega)\rangle)\rtimes\langle\tau\rangle$,
where $\U(1)=\{(1,\diag\{\lambda,\lambda,\lambda^{-2}\},\diag\{\lambda^{-1},\lambda^{-1},\lambda^{2}\}):
\lambda\in\U(1)\}$ Then $F'':=F\cap L''$ is a finite abelian subgroup of $L''$ satisfying the condition
$(*)$. Choosing an element $x\in F''-F''\cap G_0$, we may assume that $x=[(A,I,1)]\tau$ for some
$A\in\SU(2)$. Since $x[(I,-I,i)]x^{-1}=(-I,I,-i)=[(-I,-I,-1)](I,-I,i)$, one has $(L''_0)^{x}=
\Delta(\SU(2)^{A})\times\langle[(I,-I,i)]\rangle$. As $\SU(2)^{A}\cong\SU(2)$ or $\U(1)$, both have
no finite abelian subgroups satisfying the condition $(*)$, we get a contradiction. In the case of the
root system of $\frl'$ is of type $A_2^{L}$, we have $L\subset H_1$ and \[L=(\SU(3)\times T_3
\times T_3)/\langle(\omega I,\omega I,\omega I),(I,\omega I,\omega^{-1}I)\rangle\rtimes\langle\tau
\rangle.\] Let $L''=((\SU(3)\times T'_3)/\langle(I,\omega I)\rangle)\rtimes\langle\tau\rangle$. Then
$F'':=F\cap L''$ is a finite abelian subgroup of $L''$ satisfying the condition $(*)$. Since $\SU(3)$ is
a direct factor of $L''$, it does not possess any finite abelian subgroup satisfying the condition $(*)$.
We get a contradiction. In the case of the root system of $\frl'$ is of type $A_2^{S}$. We have
$L\subset H_1$ and \[L=(T_3\times\SU(3)\times\SU(3))/\langle(\omega I,\omega I,\omega I),
(I,\omega I,\omega^{-1}I)\rangle\rtimes\langle\tau\rangle.\] Let $L''=((\SU(3)\times\SU(3))/\langle
(\omega I,\omega^{-1} I)\rangle)\rtimes\langle\tau\rangle$. Then $F'':=F\cap L''$ is a finite abelian
subgroup of $L''$ satisfying the condition $(*)$. Choosing an element $x\in F''-F''\cap G_0$, we may
assume that $x=[(A,I)]\tau$ for some $A\in\SU(3)$. Since $x[(\omega I,\omega^{-1} I)]x^{-1}=
(\omega I,\omega^{-1} I)^{-1}$, one has $(L''_0)^{x}=\Delta(\SU(3)^{A})$. As $\SU(3)^{A}\cong\SU(3)$,
$\U(2)$ or $T_3$, all having no finite abelian subgroups satisfying the condition $(*)$. We get a
contradiction. In the case of the root system of $\frl'$ is of type $A_1^{L}+A_1^{S}$, we have
$L\subset H_1$ and \[L=(\U(2)\times\U(2)\times\U(2))/\langle(\omega I,\omega I,\omega I),
(I,\omega I,\omega^{-1}I)\rangle\rtimes\langle\tau\rangle.\] Let \[L''=(\SU(2)\times\SU(2)\times\SU(2)
\times\U(1)/\langle(I,-I,-I,-1),(I,I,I,\omega)\rangle)\rtimes\langle\tau\rangle,\] where
$\U(1)=\{(1,\diag\{\lambda I_2,\lambda^{-2}\},\diag\{\lambda^{-1} I_2,\lambda^{2}\}):\lambda\in\U(1)\}$.
Then $F'':=F\cap L''$ is a finite abelian subgroup of $L''$ satisfying the condition $(*)$. Since the
first $\SU(2)$ is a direct factor of $L''$, it does not possess any finite abelian subgroup satisfying
the condition $(*)$. We get a contradiction. In the case of the root system of $\frl'$ is of type
$A_2^{L}+A_1^{S}$, we have $L\subset H_1$ and \[L=(\SU(3)\times\U(2)\times\U(2))/\langle(\omega I,
\omega I,\omega I),(I,\omega I,\omega^{-1}I)\rangle\rtimes\langle\tau\rangle.\] Let
$L''=((\SU(3)\times\SU(2)\times\SU(2)\times\U(1)/\langle(I,-I,-I,-1),(I,I,I,\omega)\rangle)\rtimes
\langle\tau\rangle$, where $\U(1)=\{(1,\diag\{\lambda I_2,\lambda^{-2}\},\diag\{\lambda^{-1} I_2,
\lambda^{2}\}):\lambda\in\U(1)\}$. Then $F'':=F\cap L''$ is a finite abelian subgroup of $L''$ satisfying
the condition $(*)$. Since $\SU(3)$ is a direct factor of $L''$, it does not possess any finite abelian
subgroup satisfying the condition $(*)$. We get a contradiction. In the case of the root system of
$\frl'$ is of type $A_1^{L}+A_2^{S}$, we have $L\subset H_1$ and \[L=(\U(2)\times\SU(3)\times
\SU(3))/\langle(\omega I,\omega I,\omega I),(I,\omega I,\omega^{-1}I)\rangle\rtimes\langle\tau
\rangle.\] Let $L''=((\SU(2)\times\SU(3)\times\SU(3))/\langle(I,\omega I,\omega^{-1} I)\rangle)\rtimes
\langle\tau\rangle$. Then $F'':=F\cap L''$ is a finite abelian subgroup of $L''$ satisfying the
condition $(*)$. Since $\SU(2)$ is a direct factor of $L''$, it does not possess any finite abelian
subgroup satisfying the condition $(*)$. We get a contradiction.
\end{proof}

In \[H_3=((\Spin(10)\times\U(1))/\langle(c,i)\rangle)\rtimes\langle\tau\rangle,\] let
\[F_{13}=\langle S_3,[(-1,1)],\tau,[(e_1e_2e_3e_4,1)],[(e_1e_2e_5e_6,1)],[(e_1e_3e_5e_7,1)]\rangle,\]
where $S_3=\{\cos\theta+\sin\theta e_{9}e_{10}: \theta\in\bbR\}$.
In \[H_2=((\SU(6)\times\Sp(1))/\langle(\omega I,1),(-I,-1)\rangle)\rtimes\langle\tau\rangle,\] let
\[F_{14}=\langle S_1,S_2,[I_{4,2}],[I_{2,4}],[\diag\{I_{1,1},I_{1,1},I_{2}\}],[J_3]\tau\rangle,\]
\[F_{15}=\langle S_1,[I_{2,4}],[I_{4,2}],[\diag\{-1,1,-1,1,1,1\}],[\diag\{1,1,-1,1,-1,1\}],[J_3]\tau
\rangle,\] where $S_1=\U(1)\subset\Sp(1)$ and $S_2=\{\diag\{1,1,1,1,\lambda,\lambda^{-1}\}\in
\SU(6): |\lambda|=1\}$. Let $F_{16}=\langle T',\tau\rangle$

\begin{prop}\label{P:E6-outer-nonfinite3}
Any non-finite closed abelian subgroup of $G$ satisfying the condition $(*)$ and being not contained
in $G_0$ is conjugate to one of $F_{13}$, $F_{14}$, $F_{15}$, $F_{16}$.
\end{prop}

\begin{proof}
By Lemma \ref{L:E6-outer-nonfinite2}, the root system of $\frl'$ is one of the types: $\emptyset$,
$B_2$, $B_3$, $C_3$. In the case of the root system of $\frl'$ is of type $\emptyset$, one has
$F\sim F_{16}$. In the case of the root system of $\frl'$ is of type $B_2$, one has $L\subset H_2$
and \[L=(S(\U(4)\times\U(1)\times\U(1))\times\U(1)/\langle(\omega I,\omega,\omega,1),(-I,-1,-1,-1)
\rangle\rtimes\langle\tau\rangle,\] where $\tau[(X,\lambda_1,\lambda_2,\lambda_3)]\tau^{-1}=
[(J_2\overline{X}J_2^{-1},\overline{\lambda_2},\overline{\lambda_1},\lambda_3)]$. Let
$$L''=((\SU(4)\times\U(1))/\langle(-I,-1),(I,\omega I)\rangle)\rtimes\langle\tau\rangle,$$ where
$\U(1)=\{(\diag\{\lambda I_4,\lambda^{-2},\lambda^{-2}\},1):\lambda\in\U(1)\}$. With calculation
one shows that $Z(L'')\subset A:=Z(L)_0$. Then, $F=A\cdot F''$ and $F''=F\cap L''$ is a finite abelian
subgroup of $L''$ satisfying the condition $(*)$. Since $F''\not\subset(L'')_0$, for any
$[(X,\lambda)]\in F''\cap(L'')_0$, one has $\tau[(X,\lambda)]\tau^{-1}=[(J_2\overline{X}(J_2)^{-1},
\lambda^{-1})]$ being conjugate to $[(X,\lambda)]$ in $L''$. Hence $\lambda=\pm{1}$ or $\pm{i}$. Thus
$F''$ is conjugate to a subgroup of the form $F''=F'\times\langle[(iI,i)]\rangle$, where $F'\subset
\SU(4)\rtimes\langle\tau\rangle$. By \cite{Yu2}, Proposition 5.1, one shows that $k=0$ for the
terminology there and \[F'\sim\langle[(\diag\{-1,-1,1,1\},1)],[(\diag\{-1,1,-1,1\},1)],[(I,-1)],
[J_2]\tau\rangle.\] Therefore $F\sim F_{14}$. In the case of the root system of $\frl'$ is of type
$B_3$, we have $L\subset H_3$ and \[L\cong((\Spin(8)\times\Spin(2)\times\U(1)/\langle(c',c'',i),
(-1,-1,1)\rangle)\rtimes\langle\tau\rangle,\] where $c'=e_1e_2\cdots e_{8}$, $c''=e_{9}e_{10}$,
and $\tau[(x,y,\lambda)]\tau^{-1}=[(e_8 x e_8^{-1},y,\lambda^{-1})]$. Let $L''=((\Spin(8)\times
\U(1))/\langle(-1,-1)\rangle)\rtimes\langle\tau\rangle$. With calculation one shows that
$Z(L'')\subset A:=Z(L)_0$. Then, $F=A\cdot F''$ and $F''=F\cap L''$ is a finite abelian subgroup of
$L''$ satisfying the condition $(*)$. Since $F''\not\subset(L'')_0$, for any
$[(x,\lambda)]\in F''\cap(L'')_0$, one has $\tau[(x,\lambda)]\tau^{-1}=[(e_8 x e_8^{-1},\lambda^{-1})]$
being conjugate to $[(x,\lambda)]$ in $L''$. Hence $\lambda\in\langle i\rangle$. Thus $F''$ is
conjugate to a subgroup of the form $F''=F'\times\langle[(c',i)]\rangle$, where
$F'\subset\Spin(8)\rtimes\langle\tau\rangle\cong\Pin(8)$. Similarly as the consideration in \cite{Yu3},
Section 2, one can show that $$F'\sim\langle e_{8},e_1e_2e_3e_4,e_1e_2e_5e_6,e_1e_3e_5e_7
\rangle\ \textrm{or}\ \langle -1,c'e_8,e_1e_2e_3e_4,e_1e_2e_5e_6,e_1e_3e_5e_7\rangle.$$ In $L''$,
$-c'e_{8}=e_1e_2e_3e_4e_5e_6e_{7}=\tau$ and $e_8=[(-c',1)]\tau$. Since
$$[(c',1)]\tau=[(c',i)][(\frac{1+i}{\sqrt{2}},1)]\tau[(\frac{1+i}{\sqrt{2}},1)]^{-1},$$ both cases
lead to $F\sim F_{13}$. In the case of the root system of $\frl'$ is of type $C_3$, we have $L\subset H_2$
and \[L=((\SU(6)\times\U(1)/\langle(\omega I,1),(-I,-1)\rangle)\rtimes\langle\tau\rangle.\]
Let $L''=(\SU(6)/\langle\omega I \rangle) \rtimes\langle\tau\rangle$. Then $F=Z(L)_0\cdot F''$,
$F'':=F\cap L''$ is a finite abelian subgroup of $L''$ satisfying the condition $(*)$ and
$Z(L'')\subset Z(L)_0\subset F''$. Using \cite{Yu2}, Proposition 5.1, one can show that $k=0$ for the
terminology there and \[F''\sim\langle[-I],I_{2,4},I_{4,2},[\diag\{-1,1,-1,1,1,1\}],
[\diag\{1,1,-1,1,-1,1\}],[J_3]\tau\rangle.\] Therefore $F\sim F_{15}$.
\end{proof}

The proof of Proposition \ref{P:E6-outer-nonfinite3} indicates that each subgroup $F$ obtained
there contains an outer involution. In the case of $F$ contains an element conjugate to
$\sigma_3$, reducing to $\F_4$ one gets a unique conjugacy class of subgroups with dimension 1
from finite abelian subgroups of $\Spin(7)$, which is the one with $F_{13}$. In the case of $F$
contains an element conjugate to $\sigma_4$ and without elements conjugate to $\sigma_3$,
reducing to $\Sp(4)/\langle-I\rangle$ one gets a conjugacy class of dimension 2 from finite
abelian subgroups of $(\Sp(1)\times\Sp(1))/\langle(-1,-1)\rangle\cong\SO(4)$, which is the one
with $F_{14}$; and two conjugacy classes of dimension 1 from finite abelian subgroups of
$\SU(4)/\langle-I\rangle\cong\SO(6)$ and of $(\Sp(1)\times\Sp(2))/\langle(-1,-I)\rangle$,
the ones with $F_{15}$ and $F_{13}$ respectively.

\begin{prop}\label{P:Weyl-E6-F13 to F15}
We have $$W(F_{13})=(\mathbb{F}_2^{3}\rtimes\GL(3,\mathbb{F}_2))\times\{\pm{1}\}\times\{\pm{1}\},$$
$$W(F_{15})=(\mathbb{F}_2^{5}\rtimes S_6)\times\{\pm{1}\},$$ and there is an exact sequence
\[1\longrightarrow\mathbb{F}_2^{4}\rtimes S_4\longrightarrow W(F_{14})\longrightarrow D_4
\longrightarrow 1.\]
\end{prop}

\begin{proof}
Since $(F_{13})_0$ is a maximal torus of an $\SU(2)$ subgroup, one has $W((F_{13})_0)\cong\{\pm{1}\}$.
There is a homomorphsim $p: W(F_{13})\longrightarrow W((F_{13})_0)$, which is clearly a surjective map.
Considering in $$C_{G}((F_{13})_0)\cong((\Spin(8)\times\Spin(2)\times\U(1)/\langle(c',c'',i),(-1,-1,1)
\rangle)\rtimes\langle\tau\rangle$$ shows $\ker p=(\mathbb{F}_2^{3}\rtimes\GL(3,\mathbb{F}_2))\times
\{\pm{1}\}$, where $\{\pm{1}\}$ means the conjugation action of $c'\in\Spin(8)$, which maps $\tau$ to
$[(-1,1)]\tau$ and acts as identity on $$\langle[(-1,1)],[(e_1e_2e_3e_4,1)],[(e_1e_2e_5e_6,1)],
[(e_1e_3e_5e_7,1)]\rangle;$$ $\mathbb{F}_2^{3}\rtimes\GL(3,\mathbb{F}_2))$ means the conjugation action
on $$\langle[(-1,1)],[(e_1e_2e_3e_4,1)],[(e_1e_2e_5e_6,1)],[(e_1e_3e_5e_7,1)]\rangle$$ coming from
elements in $\Spin(7)=\Spin(8)^{\tau}$. On the other hand, since $F_{13}$ is contained in a subgroup
$(\Spin(7) \times\Spin(3)\times\U(1)/\langle(-1,1,-1),(-1,-1,1)\rangle)\rtimes\langle\tau\rangle$ with
$(F_{13})_0$ being a maximal torus of $\Spin(3)$, we actually have $W(F_{13})=(\mathbb{F}_2^{3}
\rtimes\GL(3,\mathbb{F}_2)) \times\{\pm{1}\}\times\{\pm{1}\}$.

Recall that the root system $\F_4$ has a sub-root system of type $2B_2$. Hence $(F_{14})_0$ is the
maximal torus of a subgroup isomorphic to $\Sp(2)$. Thus $W((F_{14})_0)\cong D_4$. There is a homomorphsim
$p: W(F_{14})\longrightarrow W((F_{14})_0)$. It is clear that $p$ is surjective. Consideration in
$$C_{G}((F_{14})_0)\cong(S(\U(4)\times\U(1)\times\U(1))\times\U(1)/\langle(\omega I,\omega,\omega,1),
(-I,-1,-1,-1)\rangle\rtimes\langle\tau\rangle$$ shows $$\ker p\cong\mathbb{F}_{2}^{4}\rtimes S_4.$$
Here in $L''=((\SU(4)\times\U(1))/\langle(-I,-1),(I,\omega I)\rangle)\rtimes\langle\tau\rangle$,
$$F'=F\cap L''=\langle[(\diag\{-1,-1,1,1\},1)],[(\diag\{-1,1,-1,1\},1)],[(I,-1)],[J_2]\tau\rangle.$$
The $\mathbb{F}_{2}^{4}$ comes from the conjugation action of elements $$[(\diag\{i,-i,1,1\},1)],
[(\diag\{i,1,-i,1\},1)], [(iI,1)], [(\diag\{\eta,\eta,\eta,-\eta\},\eta)],$$ where
$\eta=\frac{1+i}{\sqrt{2}}$. The $S_4$ means permutation on the diagonal matrices in $\SU(4)$.

Since $(F_{15})_0$ is a maximal torus of an $\SU(2)$ subgroup, one has $W((F_{15})_0)\cong\{\pm{1}\}$.
There is a homomorphsim $p: W(F_{15})\longrightarrow W((F_{15})_0)$. It is clear that $p$ is surjective.
Considering in $$C_{G}((F_{15})_0)\cong((\SU(6)\times\U(1)/\langle(\omega I,1),(-I,-1)\rangle)\rtimes
\langle\tau\rangle$$ shows $\ker p=(\mathbb{F}_2)^{5}\rtimes S_{6}$. The $\mathbb{F}_{2}^{5}$ comes
from the conjugation action of elements $$[(\diag\{iI_{1,1},I_4\},1)], [(\diag\{1,iI_{1,1},I_3\},1)],
[(\diag\{I_2,iI_{1,1},I_2\},1)],$$ $$[(\diag\{I_3,iI_{1,1},1\},1)], [(\diag\{I_4,iI_{1,1}\},1)].$$
The $S_6$ means permutation on the diagonal matrices in $\SU(6)$. On the other hand, since $F_{15}$ is
contained in $H_2=((\SU(6)\times\Sp(1))/\langle(\omega I,1),(-I,-1)\rangle)\rtimes\langle\tau\rangle$
with $(F_{15})_0$ being a maximal torus of $\Sp(1)$, we actually have
$W(F_{15})=(\mathbb{F}_2^{5}\rtimes S_6)\times\{\pm{1}\}$.
\end{proof}

\section{Simple Lie group of type $\bf E_7$}

Let $G=\Aut(\fre_7)$ and $\pi: \E_7\longrightarrow G$ be the adjoint homomorphism. Denote by
$c=\exp(\pi i(H'_{2}+H'_{5}+H'_{7}))\in\E_7$. Then, $o(c)=2$ and $\ker\pi=Z(\E_7)=\langle c\rangle$
(cf. \cite{Huang-Yu} Section 2).

\begin{lemma}\label{L:E7-order}
Let $F$ be a finite abelian subgroup of $G$ satisfying the condition $(*)$. For any $x\in F$,
\begin{enumerate}
\item {if $o(x)=p$ is a prime, then $p=2$ or $3$.}
\item {If $o(x)=3^{k}$, $k\geq 1$, then $k=1$ and the root system of $\frg^{x}$ is of type $A_5+A_2$.}
\end{enumerate}
\end{lemma}

\begin{proof}
Let $1\neq x\in F$ be an element of order prime to 2. Then there exists $y\in\E_7$ such that $\pi(y)=x$
and $o(y)=o(x)$. Since $y\not\sim cy$, one has $ G^{x}=(\E_7)^{y}/\langle c\rangle$ and hence it is
connected by Steinberg's theorem. Since $F$ satisfies the condition $(*)$, $G^{x}$ is semisimple. Thus
the root system of $\frg^{x}$ is one of the following types (cf. \cite{Oshima}): $A_7$, $D_6+A_1$,
$D_4+3A_1$, $A_5+A_2$, $2A_3+A_1$, $7A_1$. As it is assumed that $o(x)$ is prime to $2$, the root
system of $\frg^{x}$ is of type $A_5+A_2$. Therefore $G^{x}\cong(\SU(6)\times\SU(3))/\langle(-I,I),
(\omega I,\omega^{-1} I)\rangle$ and $o(x)=3$.
\end{proof}

In $(\SU(6)\times\SU(3))/\langle(-I,I),(\omega I,\omega^{-1}I)\rangle$, let $$F_1=\langle[(A_6,A_3)],
[(B_6,B_3)],[(I,\omega I)]\rangle,$$ where $A_{3}=\diag\{1,e^{\frac{2\pi i}{3}},e^{\frac{4\pi i}{3}}\}$,
$A_{6}=\diag\{e^{\frac{\pi i}{6}},e^{\frac{3\pi i}{6}},e^{\frac{5\pi i}{6}},e^{\frac{7\pi i}{6}},
e^{\frac{9\pi i}{6}},e^{\frac{11\pi i}{6}}\}$, \[B_3=\left(\begin{array}{ccc}0&1&0\\0&0&1\\1&0&
0\end{array}\right),\quad B_6=e^{\frac{\pi i}{6}}\left(\begin{array}{cccccc}0&1&0&0&0&0\\0&0&1&0&0
&0\\0&0&0&1&0&0\\0&0&0&0&1&0\\0&0&0&0&0&1\\1&0&0&0&0&0\end{array}\right).\]

\begin{prop}\label{P:E7-finite1}
Let $F$ be a finite abelian subgroup of $G$ satisfying the condition $(*)$. If $F$ is not a 2-group,
then $F\sim F_1$.
\end{prop}

\begin{proof}
By the proof of Lemma \ref{L:E7-order}, $F$ contains an element $x$ of order 3 and \[F\subset G^{x}
\cong(\SU(6)\times\SU(3))/\langle(-I,I),(\omega I,\omega^{-1}I)\rangle.\] For any
$(\lambda_1 I,\lambda_2 I)\in\langle(-I,I),(\omega I,\omega^{-1}I)\rangle$, one has $\lambda_2=
\lambda_1^{2}$, Using \cite{Yu2}, Proposition 2.1, one gets $F\sim F_1$.
\end{proof}

\begin{prop}\label{P:Weyl-E7-F1}
We have $W(F_1)\cong\SL(3,\mathbb{F}_3)\times\SL(2,\mathbb{F}_2)$.
\end{prop}

\begin{proof}
Choosing an element $x$ of $F_1$ with order $3$, consideration in $G^{x}$ shows that
$\Stab_{W(F_1)}(x)\cong(\mathbb{F}_3^{2}\rtimes\SL(2,\mathbb{F}_3))\times\SL(2,\mathbb{F}_2)$.
Apparently $W(F_1)$ acts transitively on the set of elements of $F_1$ with order $3$. Therefore
$W(F_1)\cong\SL(3,\mathbb{F}_3)\times\SL(2,\mathbb{F}_2)$.
\end{proof}

In $G^{\sigma_2}\cong((\E_6\times\U(1))/\langle(c',e^{\frac{2\pi i}{3}})\rangle)\rtimes\langle
\omega\rangle$ (cf. \cite{Huang-Yu}, Table 2), let $F_2$, $F_3$, $F_4$, $F_5$, $F_6$, $F_7$ be finite
abelian 2-subgroups of $\E_6\rtimes\langle\omega\rangle$ with the image of projection to
$\Aut(\fre_6)$ being the subgroups $F_6$, $F_8$, $F_9$, $F_{10}$, $F_{11}$, $F_{12}$ of $\Aut(\fre_6)$
respectively. Here $c'$ is an generator of $\E_6$, $(\fre_6\oplus i\bbR)^{\omega}=\frf_4\oplus 0$.

\begin{prop}\label{P:E7-finite2}
If $F$ is a finite abelian 2-subgroup of $G$ satisfying the condition $(*)$ and containing an element conjugate to
$\sigma_2$, then $F$ is conjugate to one of $F_2$, $F_3$, $F_4$, $F_5$, $F_6$, $F_7$.
\end{prop}

\begin{proof}
We may and do assume that $\sigma_2\in F$. Then, \[F\subset G^{\sigma_2}\cong((\E_6\times\U(1))/\langle(c',
e^{\frac{2\pi i}{3}})\rangle)\rtimes\langle\omega\rangle.\] One has $F\not\subset(G^{\sigma_2})_0$ as $F$
satisfies the condition $(\ast)$ and $(G^{\sigma_2})_0$ has a one-dimensional center. Given $[(x,\lambda)]\in
F\cap(G^{\sigma_2})_0$, one has $\omega[(x,\lambda)]\omega^{-1}=[(\cdot,\lambda^{-1})]$ being conjugate to
$[(x,\lambda)]$ in $(G^{\sigma_2})_0$. Hence $\lambda=\pm{1}$. Thus $F$ is conjugate to a subgroup of the form
$F=F'\times\langle[(1,-1)]\rangle$ with $F'$ a finite abelian subgroup of $\E_6\rtimes\langle\omega\rangle$
satisfying the condition $(\ast)$ and being not contained in $\E_6$. By Lemma \ref{L:E6-outer finite1} and
Proposition \ref{P:E6-outer finite2}, $F$ is conjugate to one of $F_2$, $F_3$, $F_4$, $F_5$, $F_6$, $F_7$.
\end{proof}


Among $F_2$, $F_3$, $F_4$, $F_5$, $F_6$, $F_7$, the maximal abelian ones are $F_2$, $F_3$, $F_6$, $F_7$.

\begin{prop}\label{P:Weyl-E7-F2 to F7}
There is an exact sequence
\[1\longrightarrow\Hom'(\mathbb{F}_2^{3},\mathbb{F}_2^{3})\times\mathbb{F}_2\longrightarrow W(F_2) \longrightarrow(\mathbb{F}_2)^{2}\rtimes P(2,1,\mathbb{F}_2)\longrightarrow 1,\] and we have
\[W(F_3)\cong\Hom((\mathbb{F}_2)^{2},(\mathbb{F}_2)^{2})\rtimes(\GL(2,\mathbb{F}_2)
\times P(2,3,\mathbb{F}_2)),\]
\[W(F_4)\cong((\bbF_2)^{5}\rtimes\Hom((\bbF_2)^{3},(\bbF_2)^{2}))\rtimes(\GL(2,\bbF_2)
\times\GL(2,\bbF_2)),\]
\[W(F_5)\cong((\bbF_2)^{6}\rtimes\Hom((\bbF_2)^{5},\bbF_2))\rtimes\Sp(2,\mathbb{F}_2),\]
\[|W(F_{6})|=3\times 2^{15},\]
\[W(F_7)\cong(\bbF_2)^{7}\rtimes\Sp(3,\mathbb{F}_2).\]
\end{prop}

\begin{proof}
The abelian subgroups $F_3$, $F_4$, $F_5$, $F_7$ are elementary abelian 2-subgroups. The description
of their Weyl groups is given in \cite{Yu}, Proposition 7.26. For $F_2$, let $K=\{x\in F_2: x^2=1\}$
and $K'=\{x^2: x\in F_2\}$. Then, $K'\subset K$, $\rank K=4$, and $\rank K'=3$. From the description
of the subgoup $F_6$ in \cite{Yu3}, one sees that $K$ is conjugate to the elementary abelian
$2$-subgroup $F'_{2,1}$ of $G$ in the terminology of \cite{Yu}. Hence $W(K)\cong(\mathbb{F}_2)^{2}
\rtimes P(2,1,\mathbb{F}_2)$ by \cite{Yu}, Proposition 7.26. There is a homomorphism $p: W(F_2)
\longrightarrow W(K)$. Using the fact that $W(F_2)$ acts transitively on the set of elements of $F_2$
being conjugate to $\sigma_2$, one gets that $p$ is surjective. Similarly as the proof Proposition
\ref{P:Weyl-E6-F6}, one shows that $\ker p\cong\Hom'(\mathbb{F}_2^{3},\mathbb{F}_2^{3})\times
\mathbb{F}_2$, where $\Hom'(\mathbb{F}_2^{3},\mathbb{F}_2^{3})$ comes from conjugation action of
elements in $\E_6\subset G^{\sigma_2}$ and $\mathbb{F}_2$ comes from the conjugation action of
$[(1,i)]\in G^{\sigma_2}$. For $F_6$, let $K=\{x\in F_{6}: x^2=1\}$ and $K'=\{x^2: x\in F_{6}\}$.
Then, $K'\cong\mathbb{F}_2$. Denote by $z$ a generator of $K'$. Then $W(F_6)$ fixes $z$ and stabilizes
$K$. Hence there is a homomorphism $p: W(F)\longrightarrow\Stab_{W(K)}(z)$. Similar argument as
\cite{Yu3}, Proposition 3.3 shows that $\ker p\cong(\mathbb{F}_2)^2$ and $p$ is surjective. Since
$K\sim F_{4}$, the orbit $W(K)z$ has three elements. One has $|W(K)|=|W(F_4)|=32\times 64
\times 6^2=3^2\times 2^{13}$ and hence $|W(F_{6})|=\frac{4}{3}|W(K)|=3\times 2^{15}$.
\end{proof}

In $G^{\sigma_3}\cong(\SU(8)/\langle iI\rangle)\rtimes\langle\omega\rangle$ (\cite{Huang-Yu}, Table 2), let
\[F_8=\langle\sigma_3,[A_4],[B_4],[\diag\{I_{1,1},I_{1,1},I_{1,1},I_{1,1}\}],
[\diag\{J'_1,J'_1,J'_1,J'_1\}]\rangle,\] \begin{eqnarray*}F_9&=&\langle\sigma_3,[I_{4,4}],[J'_4],
[\diag\{I_{2,2},I_{2,2}\}],[\diag\{J'_2,J'_2\}],[\diag\{I_{1,1},I_{1,1},I_{1,1},I_{1,1}\}],\\&&
[\diag\{J'_1,J'_1,J'_1,J'_1\}]\rangle,\end{eqnarray*}
\[F_{10}=\langle\sigma_3,[J_4]\omega,[I_{4,4}],[\diag\{I_{2,2},I_{2,2}\}],[\diag\{I_{1,1},I_{1,1},I_{1,1},I_{1,1}\}]\rangle,\]
\begin{eqnarray*} F_{11}&=&\langle\sigma_3,[J_4]\omega,[I_{4,4}],[\diag\{I_{2,2},I_{2,2}\}],
[\diag\{I_{1,1},I_{1,1},I_{1,1},I_{1,1}\}],\\&&[\diag\{J'_{1},J'_{1},J_{1},J_{1}\}]\rangle.
\end{eqnarray*}

Here $A_4=\diag\{I_2,iI_2,-I_2,-iI_2\}$ and $B_4=\left(\begin{array}{cccc}0_2&I_2&0_2&0_2\\0_2
&0_2&I_2&0_2\\0_2&0_2&0_2&I_2\\I_2&0_2&0_2&0_2\\\end{array}\right)$.

\begin{prop}\label{P:E7-finite5}
If $F$ is a finite abelian 2-subgroup of $G$ satisfying the condition $(*)$, containing an element
conjugate to $\sigma_3$, and without elements conjugate to $\sigma_2$, then $F$ is conjugate to
one of $F_8$, $F_9$, $F_{10}$, $F_{11}$.  
\end{prop}

\begin{proof}
We may and do assume that $\sigma_3\in F$. Then,  $F\subset G^{\sigma_3}\cong(\SU(8)/\langle iI
\rangle)\rtimes\langle\omega\rangle$, where $\sigma_3=[\frac{1+i}{\sqrt{2}}I]$ and $\omega[X]
\omega^{-1}=[J_4\overline{X}J_4^{-1}]$ for $X\in\SU(8)$. Since it is assumed that $F$ does not
contain any element conjugate to $\sigma_2$, $F$ has no elements conjugate to $[I_{2,6}]$ or
$\omega$ in $G^{\sigma_3}$ (cf. \cite{Huang-Yu}, Page 414). We consider the cases of
$F\subset(G^{\sigma_3})_0$ and $F\not\subset(G^{\sigma_3})_0$ separately. In the case of
$F\subset(G^{\sigma_3})_0$, by \cite{Yu}, Proposition 2.1, one gets $F\sim F_8$ or $F\sim F_9$.
In the case of $F\not\subset(G^{\sigma_3})_0$, by \cite{Yu2}, we have integers $k$, $s_0$ with
$8=2^{k}\cdot s_0$. The condition of $F$ has no elements conjugate to $[I_{2,6}]$ or $\omega$ in
$G^{\sigma_3}$ simplies the consideration. From this, first any element of $B_{F}$ (cf. \cite{Yu},
Section 5) is conjugate to $[I_{4,4}]$ and hence the conjugacy class of $B_{F}$ is determined by
$\rank B_{F}$, which is an integer at most 3. If $k=0$, then $F\sim F_{10}$. If $k=1$, we have
four maps $\mu_1,\mu_2,\mu_3,\mu_4: F\cap(G^{\sigma_3})_0/B_{F}\longrightarrow\{\pm{1}\}$ compatible
with an antisymmetric bimultiplicative function $m$ on $F\cap(G^{\sigma_3})_0$. In the case of $\mu_1=\mu_2\neq\mu_3=\mu_4$, one has $F\sim F_{11}$. In the other cases one can show that $F$
contains an element conjugate to $[I_{2,6}]$ or $\omega$, which is a contradiction. In the case
of $k=2$ or $3$, $F$ contains an element conjugate to $\omega$, which is a contradiction.
\end{proof}

Among $F_8$, $F_9$, $F_{10}$, $F_{11}$, only $F_{8}$ is a maximal abelian subgroup.

\begin{prop}\label{P:Weyl-E7-F8 and F9}
There is an exact sequence \begin{eqnarray*}1&\rightarrow&\Hom(\mathbb{F}_2^2,\mathbb{F}_2^2)
\rightarrow W(F_8)\\&\rightarrow&(\bbF_{2}^{4}\rtimes\Hom(\bbF_{2}^{2},\bbF_2^{2}))\rtimes(\GL(2,\bbF_2)
\times\GL(2,\bbF_2))\rightarrow 1.\end{eqnarray*}
\end{prop}

\begin{proof}
Let $K=\{x\in F_8: x^2=1\}$ and $K'=\{x^2: x\in F_8\}$. Then, $K\sim F''_{2,1}$ and $K'=A_{K}$ in
the terminology of \cite{Yu}. Hence $W(K)\cong(\bbF_{2}^{4}\rtimes\Hom(\bbF_{2}^{2},\bbF_2^{2}))
\rtimes(\GL(2,\bbF_2)\times\GL(2,\bbF_2))$ and $K'$ is stable under $W(K)$. There is a homomorphism
$p: W(F_8)\longrightarrow W(K)$. Using the fact that the set of elements of $K$ conjugate to $\sigma_3$
is transitive under the action of $W(F_8)$, one can show that $p$ is surjective by considering in
$G^{\sigma_3}$. Similarly one shows $\ker p=\Hom(F_8/K, K')\cong\Hom(\bbF_{2}^{2},\bbF_2^{2})$ by
considering in $G^{\sigma_3}$.
\end{proof}

Recall that $\Gamma_1$ is a Klein four subgroup of $G$ with non-identity elements all conjugate
to $\sigma_1$ and (cf. \cite{Huang-Yu}, Table 6) \[G^{\Gamma_1}\cong((\SU(6)\times\U(1)\times
\U(1))/\langle(e^{\frac{2\pi i}{3}}I,e^{\frac{-2\pi i}{3}},1),(-I,1,1)\rangle)\rtimes\langle z
\rangle,\] where $\Gamma_1=\langle[(I,-1,1)],[(I,1,-1)]\rangle$ and $\Ad(z)[(X,\lambda_1,\lambda_2)]=[(J_3\overline{X}J_3^{-1},\lambda_1^{-1},\lambda_2^{-1})]$.

\begin{lemma}\label{L:E7-finite2}
Let $F$ be a finite abelian 2-subgroup of $G$ satisfying the condition $(*)$ and without elements
conjugate to $\sigma_2$ or $\sigma_3$. Then $F$ contains no Klein four subgroups conjugate to
$\Gamma_1$.
\end{lemma}

\begin{proof}
Suppose the conclusion does not hold. We may and do assume that $\Gamma_1\subset F$. Then,
\[F\subset G^{\Gamma_1}\cong((\SU(6)\times\U(1)\times\U(1))/\langle(e^{\frac{2\pi i}{3}}I,
e^{\frac{-2\pi i}{3}},1),(-I,1,1)\rangle)\rtimes\langle z\rangle.\] Thus $F\not\subset
(G^{\Gamma_1})_0$ and it is conjugate to a subgroup of the form $F'\times\Gamma_1$, where $F'$
is a subgroup of $(\SU(6)/\langle-I\rangle)\rtimes\langle z\rangle$ satisfying the condition
$(\ast)$. Applying \cite{Yu2}, Proposition 5.1, we get integers $k$, $s_0$ with $6=2^{k}\cdot s_0$.
Then, $k=0$ or $1$. In the case of $k=0$, one can show that $F'$ contains an element conjugate
to $[J_3]z$. In the case of $k=1$, we have three maps $\mu_1,\mu_2,\mu_3: F'\cap(G^{\Gamma_1})_0
\rightarrow\{\pm{1}\}$ compatible with a bimultiplicative function $m'$ on $F'\cap(G^{\Gamma_1})_0$.
Considering determinat one can show that $\mu_1=\mu_2=\mu_3$, or $\mu_1$, $\mu_2$, $\mu_3$ are
distinct. If $\mu_1=\mu_2=\mu_3$, one shows that $F'$ contains an element conjugate to $[J_3]$.
If $\mu_1$, $\mu_2$, $\mu_3$ are distinct, one shows that $F'$ contains an element conjugate to
$[J_3]z$. Thus $F$ contains an element conjugate to $[(J_3,1,1)]$ or $[(J_3,1,1)]z$. From
\cite{Huang-Yu}, Page 413, one can show that $[(J_3,1,1)]$ and $[(J_3,1,1)]z$ are conjugate to
$\sigma_3$ in $G$, which contradicts the assumption that $F$ contains no elements conjugate to
$\sigma_3$.
\end{proof}

\begin{lemma}\label{L:E7-finite3}
If $F$ is a finite abelian 2-subgroup of $G$ satisfying the condition $(*)$ and with an element
conjugate to $\sigma_1$, then $y^{4}=1$ for any $y\in F$.
\end{lemma}

\begin{proof}
Without loss of generality we assume that $\sigma_1\in F$. Then, $F\subset G^{\sigma_1}=(\Spin(12)
\times\Sp(1))/\langle(-c,-1),(c,1)\rangle$. By \cite{Yu2}, Proposition 2.1, for any $y=[(x,\lambda)]
\in F$ with $x\in\Spin(12)$ and $\lambda\in\Sp(1)$, one has $\lambda^2=\pm{1}$. By \cite{Yu2}
Proposition 3.2, by considering the image of the projection of $G^{\sigma_1}$ to
$\SO(12)/\langle-I\rangle$, one gets integers $k$, $s_0$ with $12=2^{k}\cdot s_0$. In the case of
$k=0$, one has $x^2\in\{\pm{1}\}$. In the case of $k=2$, one has $x^2\sim\pm{e_1e_2\cdots e_{k}}$,
$k=0$, $4$, $8$ or $12$. In the case of $k=1$, $x$ is conjugate to an element of the form
$x=x_1x_2x_3x_4x_5x_6$, where $(x_1,x_2,x_3,x_4,x_5,x_6)\in\Pin(2)^{6}\cap\Spin(12)$,
$x_{i}=\pm{e_{i}}$ or $\pm{\frac{1\pm{}e_{2i-1}e_{2i}}{\sqrt{2}}}$ for each $1\leq i\leq 6$, or
$x_{i}\in\{\pm{1},\pm{e_{2i-1}e_{2i}}\}$ for each $1\leq i\leq 6$. In the first case, using the
condition of $x\in\Spin(12)$, one gets $x^2\sim\pm{e_1e_2\cdots e_{k}}$, $k=0$, $4$, $8$ or $12$.
In the latter case, $x^2=\pm{1}$. Hence in any case $x^4=1$ and $\lambda^4=1$. Therefore $y^{4}=1$
for any $y\in F$.
\end{proof}

In $G^{\sigma_1}=(\Spin(12)\times\Sp(1))/\langle(-c,-1),(c,1)\rangle$, let \[F_{12}=\langle[(\delta_1,\mathbf{i})],
[(\delta_2,\mathbf{j})],[(e_1e_2e_3e_4e_5e_6,1)]\rangle,\] where \[\delta_1=\frac{1+e_1e_2}{\sqrt{2}}e_3e_5
\frac{1+e_7e_8}{\sqrt{2}}e_{9}e_{11}\] and \[\delta_2=e_1\frac{1+e_3e_4}{\sqrt{2}}\frac{e_5+e_6}{\sqrt{2}}
e_7\frac{1+e_9e_{10}}{\sqrt{2}}\frac{e_{11}-e_{12}}{\sqrt{2}}.\]

\begin{prop}\label{P:E7-finite3}
If $F$ is a finite abelian 2-subgroup of $G$ satisfying the condition $(*)$ an without elements conjugate to
$\sigma_2$ or $\sigma_3$, then $F\sim F_{12}$.
\end{prop}

\begin{proof}
We may and do assume that $\sigma_1\in F$. Then, \[F\subset G^{\sigma_1}=(\Spin(12)\times\Sp(1))/\langle(-c,-1),
(c,1)\rangle.\] Since $F$ contains no elements conjugate to $\sigma_2$ or $\sigma_3$, it contains no Klein four
subgroups conjugate to $\Gamma_1$ by Lemma \ref{L:E7-finite2}. By \cite{Huang-Yu}, Page 413, any involution in
$F-\langle\sigma_1\rangle$ is conjugate to $[(e_1e_2e_3e_4,1)]$ in $G^{\sigma_1}$.
By Lemmas \ref{L:E7-finite3}, $x^4=1$ for any
$x\in F$. Let $A_1=\{x\in F: x^2=1\}$, $A_2=\{x^2: x\in F\}$, $r_{1}=\rank A_{1}$, $r_{2}=\rank A_{2}$. Then,
there is an exact sequence $1\rightarrow A_1\rightarrow F\rightarrow A_2\rightarrow 1$ with the map
$F\longrightarrow A_2$ given by $x\mapsto x^2$, $x\in F$; $A_1$ is conjugate to a subgroup of $\langle\sigma_1,
[(e_1e_2e_3e_4,1)],[(e_5e_6e_7e_8,1)]\rangle$; $r_2\leq r_1\leq 3$. From the covering $\pi: \E_7\longrightarrow G$,
we define a bimultiplicative function $m: F\times F\longrightarrow\langle c\rangle$ by $[x',y']=m(x,y)$ for
$x=\pi(x'),y=\pi(y')\in F$, $x',y'\in\E_7$. Since $G^{x}$ is connected for any $x\in A_1$, one has $A_1\subset
\ker m$. We consider the case of $\ker m=A_1$ and the case of $\ker m\neq A_1$ separately. In the case of $\ker m=
A_1$. Since $F/\ker m$ is an elementary abelian 2-subgroup of even rank, $F\neq A_1$ due to $A_1$ does not satisfy
the condition $(\ast)$, and $F/A_1\cong A_2$ is an elementary abelian 2-subgroup of rank at most 3, one has
$A_2\cong F/A_1=F/\ker m$ being an elementary abelian 2-subgroup of rank 2.  Hence $(r_1,r_2)=(2,2)$ or $(3,2)$.
Considering the image $F'$ of the projection of $F$ to $\Spin(12)/\langle-1,-c\rangle\cong\SO(12)/\langle-I\rangle$,
then $|F'|\leq |F|/2\leq 2^{4}$. By \cite{Yu2}, Proposition 3.1, we have integers $k$, $s_0$ associated to $F'$ with
$12=2^{k}\cdot s_0$ and $|F'|\geq 2^{2k}s_0=12\cdot 2^{k}$. Hence $k=0$. That means $F$ is conjugate to a subgroup
of the image of $\langle e_1e_2,e_1e_3,\dots,e_1e_{12}\rangle\times\langle\textbf{i},\textbf{j}\rangle$ in
$G^{\sigma_1}$. By this one gets $A_2\subset\langle\sigma_2\rangle$, which contradicts $\rank A_2=2$. In the case
of $\ker m\neq A_1$, there is an element $x$ in $\ker m$ of order 4. We may and do assume that $x^2=\sigma_1$. Then
$x$ is conjugate to one of $[(e_1e_2,1)]$, $[(e_1e_2e_4e_4e_5e_6,1)]$, $[(e_1\Pi e_1,1)]$, $[(-e_1\Pi e_1,1)]$,
$[(e_1e_2e_3e_4,\textbf{i})]$, $[(1,\textbf{i})]$, $[(\Pi,\textbf{i})]$. Choosing $x'\in\pi^{-1}(x)$, then
$F\subset\pi(\E_7^{x'})$ since $x\in\ker m$. By Steinberg's theorem $\E_7^{x'}$ is connected and hence
$F\subset(G^{\sigma_1})^{x}_0$. Except $[(e_1e_2e_4e_4e_5e_6,1)]$, any other choice of the conjuagacy class of
$x$ leads to a non-semisimple centralizer in $G^{\sigma_1}$. Therefore $x\sim[(e_1e_2e_4e_4e_5e_6,1)]$ and
\begin{eqnarray*}F&\subset&(G^{\sigma_1})^{x}_0=(\Spin(6)\times\Spin(6)\times\Sp(1))/\langle(c_{6},-c_{6},-1),
(-1,1,-1)\rangle\\&&\cong(\SU(4)\times\SU(4)\times\Sp(1))/\langle(iI,-iI,-1),(I,-I,-1)\rangle.\end{eqnarray*}
One can show that \[F\sim\langle[(A_4,A_4,\textbf{i})],[(B_4,B_4^{-1},\textbf{j})],[(iI,I,1)]\rangle,\] where
$A_4=\frac{1+i}{\sqrt{2}}\diag\{1,i,-1,-i\}$ and $B_4=\frac{1+i}{\sqrt{2}}\left(\begin{array}{cccc}0&1&0&0
\\0&0&1&0\\0&0&0&1\\1&0&0&0\end{array}\right)$. This means $F\sim F_{12}$.
\end{proof}

\begin{prop}\label{P:Weyl-E7-3C_4}
One has $C_{G}(F_{12})/F_{12}\cong(\mathbb{F}_2)^3$. There is an exact sequence $$1\rightarrow
\Hom'(\mathbb{F}_2^3,\mathbb{F}_2^3)\rightarrow W(F_{12})\rightarrow\GL(3,\mathbb{F}_2)\rightarrow 1,$$
where $\Hom'(\mathbb{F}_2^3,\mathbb{F}_2^3)=\{f\in\Hom(\mathbb{F}_2^3,\mathbb{F}_2^3): \tr f=0\}$.
\end{prop}

\begin{proof}
Let $K=\{x\in F_{12}: x^2=1\}$. Then, there is a homomorphism $p: W(F_{12})\longrightarrow W(K)$. From the
construction of $F_{12}$ in $G^{\sigma_1}$, one sees that the bimultiplicative function $m$ on $F_{12}$ is
trivial. Hence by the proof of Proposition \ref{P:E7-finite3}, $W(F_{12})$ acts transitively on the set of
elements of order $4$ in $F_{12}$. On the other hand, one can calculate $\Stab_{W(F_{12})}(x)$ for an element
$x$ of $F_{12}$ with order $4$ by consideration in $G^{x}$. For this, we could assume that
$x=[(e_1e_2e_3e_4e_5e_6,1)]$. Denote by $x'_1=(\delta_1,\mathbf{i})$ and $x'_2=(\delta_2,\mathbf{j})$.
Let $$\Pi=\frac{1+e_{1}e_{7}}{\sqrt{2}}\frac{1+e_{2}e_{8}}{\sqrt{2}}\frac{1+e_{3}e_{9}}{\sqrt{2}}\frac{1+e_4e_{10}}{\sqrt{2}}\frac{1+e_{5}e_{11}}{\sqrt{2}}\frac{1-e_{6}e_{12}}{\sqrt{2}}.$$ Then, $[(\Pi,1)]\in C_{G}(K)$
and $C_{G}(x)$ is generated by $(C_{G}(x))_0$ and $[(\Pi,1)]\in C_{G}(K)$. By this $\Stab_{W(F_{12})}(x)$
can be calculated in $(C_{G}(x))_0$. By calculation in $(C_{G}(x))_0$ we get
$\Stab_{W(F_{12})}(x)\cong\Hom((\bbZ/4\bbZ)^{2},\bbZ/4\bbZ)\rtimes\SL(2,\bbZ/4\bbZ)$. By this
the homomorphism $p$ is surjective and moreover $\ker p=\Hom'(\mathbb{F}_2^3,\mathbb{F}_2^3)$.

For any $x=\pi(x')\in F_{12}$ and $g=[g']\in C_{G}(F_{12})$, $x,g'\in\E_7$, one has $g'x'g'^{-1}x'^{-1}=1$ or
$c$. Since $G^{x}$ is connected while $x\in K$ and the bimultiplicative function $m$ on $F_{12}$ is trivial,
this defines a homomorphism $\phi: C_{G}(F_{12})/F_{12}\longrightarrow\Hom(F_{12}/K,\mathbb{F}_2)$. Here we
identify $\langle c\rangle$ with $\mathbb{F}_2$. We show that $\phi$ is surjective and injective, therefore
one has $C_{G}(F_{12})/F_{12}\cong(\mathbb{F}_2)^3$. Suppose that $\phi(g)=0$ for some $g\in N_{G}(F_{12})$.
As in the above, we denote by $x=[(e_1e_2e_3e_4e_5e_6,1)]$, $x'_1=(\delta_1,\mathbf{i})$ and $x'_2=(\delta_2,
\mathbf{j})$. Since $\phi(g)=0$, one has $g\in (G^{x})_0\cong(\SU(4)\times\SU(4)\times\Sp(1))/\langle
(iI,-iI,-1),(I,-I,-1)\rangle$. Denote by $g=[g']$ for some $g'\in \SU(4)\times\SU(4)\times\Sp(1)$. Then,
$gx'_{i}g^{-1}(x'_{i})^{-1}\in\langle(iI,-iI,-1)\rangle$, $i=1,2$. Therefore $g\in F_{12}$ and hence $\phi$ is
injective. Let $g'=(e_1e_2,\textbf{i})\in G^{\sigma_1}$. Then $[g',x']=[g',x'_1]=1$ and $[g',x'_2]=c$. Hence
$\phi\neq 0$. As $\phi$ is $W(F_{12})$ invariant, one sees that $\phi$ is surjective.
\end{proof}

\begin{remark}
From the construction of $F_{12}$ in $G^{\sigma_1}$, one sees that $[(e_1e_2,\textbf{i})]\in C_{G}(F_{12})$.
From \cite{Huang-Yu}, Page 413, one has $[(e_1e_2,\textbf{i})]\sim\sigma_2$ in $G$. By this, one sees that
$F_{12}$ is conjugate to a subgroup of $F_2$, which in turn corresponds to the subgroup $F_6$ of $\Aut(\fre_6)$
in Section 3. As a subgroup of $F_2$, $F_{12}$ is not stable under the action of $W(F_2)$. In the terminology
of \cite{Yu}, one has $\{x\in F_{12}: x^2=1\}=F'_3$ and $\{x\in F_{2}: x^2=1\}=F'_{2,1}$. As a subgroup of
$F'_{2,1}$, $F'_{3}$ is also not stable under the action of $W(F'_{2,1})$.

By construction, the subgroup $F_2$ is related to the subgroup $F_6$ of $\Aut(\fre_6)$. On the other hand,
the iamge of projection of $F_{12}$ to $\Spin(12)/\langle-1,c\rangle$ is conjugate to the image of projection
of the subgroup $F_1$ of $\Spin(12)/\langle c\rangle$ to it.
\end{remark}



\smallskip

In $G$, let $L_1$, $L_2$, $L_3$, $L_4$, $L_5$, $L_6$ be Levi subgroups of $G$ with root systems $3A_1$,
$D_4$, $D_4+A_1$, $D_5+A_1$, $D_6$, $E_6$ respectively. Then, $$(L_1)_{s}\cong\Sp(1)^3/\langle(-1,-1,-1)
\rangle,$$ $$(L_2)_{s}\cong\Spin(8),$$ $$(L_3)_{s}\cong(\Spin(8)\times\Sp(1))/\langle(-1,-1)\rangle,$$
$$(L_4)_{s}\cong(\Spin(10)\times\Sp(1))/\langle(-1,-1)\rangle,$$ $$(L_5)_{s}\cong\Spin(12)/\langle c\rangle,$$
$$(L_6)_{s}\cong\E_6.$$ In each case, one has $Z((L_{i})_{s})\subset Z(L_{i})_0$.
In $L_1$, let $$F_{13}=\langle Z(L_1)_0,[(\textbf{i},\textbf{i},\textbf{i})],[(\textbf{j},\textbf{j},
\textbf{j})]\rangle.$$
In $L_2$, let $$F_{14}=\langle Z(L_2)_0,[e_1e_2e_3e_4],[e_1e_2e_5e_6],[e_1e_3e_5e_7]\rangle.$$
In $L_3$, let $$F_{15}=\langle Z(L_3)_0,[(e_1e_2e_3e_4,1)],[(e_1e_2e_5e_6,1)],[(e_1e_3e_5e_7,\textbf{i})],
[(e_1e_2,\textbf{j})]\rangle.$$
In $L_4$, let $$F_{16}=\langle Z(L_4)_0,[(e_1e_2e_3e_4,1)],[(e_1e_2e_5e_6,1)],[(e_1e_3e_5e_7,1)],
[(e_8e_9,\textbf{i})],[(e_8e_{10},\textbf{j})]\rangle.$$
In $L_5$, let $F_{17}$ be \begin{eqnarray*}&&\langle Z(L_5)_0, [e_1e_2e_3e_4e_5e_6],
[e_1e_3\frac{1+e_5e_6}{\sqrt{2}}e_7e_9\frac{1+e_{11}e_{12}}{\sqrt{2}}],\\&&
[\frac{e_1+e_2}{\sqrt{2}}\frac{1+e_3e_4}{\sqrt{2}}e_5\frac{e_7+e_8}{\sqrt{2}}
\frac{1+e_9e_{10}}{\sqrt{2}}e_{11}]\rangle,\end{eqnarray*}
\[F_{18}=\langle[e_1e_2e_3e_4],[e_5e_6e_7e_8],[e_3e_4e_5e_6],[e_7e_8e_9e_{10}],
[e_1e_3e_5e_7e_9e_{11}],[-1]\rangle,\] and
\[F_{19}=\langle[e_1e_2e_3e_4],[e_5e_6e_7e_8],[e_3e_4e_5e_6],[e_7e_8e_9e_{10}],[e_1e_3e_5e_7e_9e_{11}],
[\delta],[-1] \rangle,\] where \[\delta=\frac{1+e_1e_2}{\sqrt{2}}\frac{1+e_3e_4}{\sqrt{2}}
\frac{1+e_5e_6}{\sqrt{2}}\frac{1+e_7e_8}{\sqrt{2}}\frac{1+e_9e_{10}}{\sqrt{2}}\frac{1+e_{11}e_{12}}{\sqrt{2}}.\]
In $L_6$, let $$F_{20}=Z(L_6)_0\times F',$$ where $F'$ is an elementary abelian 3-subgroup of $\E_6$ projecting
isomorphically to the subgroup $F_2$ of $\Aut(\fre_6)$ in Section 2. Among the subgroups $F_{13}$, $F_{14}$,
$F_{15}$, $F_{16}$, $F_{17}$, $F_{18}$, $F_{19}$, $F_{20}$, the maximal abelian ones are $F_{13}$, $F_{14}$,
$F_{15}$, $F_{16}$, $F_{17}$, $F_{19}$, $F_{20}$.

\begin{lemma}\label{L:E7-nonfinite}
Let $F$ be a non-finite closed abelian subgroup of $G$ satisfying the condition $(*)$. Then $F$ is either a
maximal torus, or is conjugate to one of $F_{13}$, $F_{14}$, $F_{15}$, $F_{16}$, $F_{17}$, $F_{18}$, $F_{19}$,
$F_{20}$.
\end{lemma}

\begin{proof}
Let $\fra=\Lie F\otimes_{\bbR}\bbC$, $L=C_{G}(\fra)$ and $L_{s}=[L,L]$. By Lemma \ref{L:center}, $F=Z(L)_0\cdot F'$
and $F':=F\cap L_{s}$ is a finite abelian subgroup of $G$ satisfying the condition $(*)$. We may assume that $F$ is
not a maximal torus and hence the root system of $L$ is not $\emptyset$. Then, $L_{s}$ does not
possess a direct factor of type $\bf A$. We may assume that a subset $Y$ of a simple system of the root system of $
G$ is a simple system of $L$. By Lemma \ref{L:Levi} and the fact that $c=\exp(\pi i(H'_2+H'_5+H'_7))$, one has $L_{s}$
being not simply connected if and only if $\alpha_2,\alpha_5,\alpha_7\in Y$. In the case of $L_{s}$ is not simply
connected, due to $L_s$ has a finite abelian subgroup satisfying the condition $(\ast)$ the root system of $L$ is
one of the types $3A_1$, $A_3+A_1$, $D_4+A_1$, $D_5+A_1$, $D_6$. Moreover, in the $A_3+A_1$ case,
$L_s\cong(\SU(4)\times\Sp(1))/\langle(-I,-1)\rangle$, which neither have finite abelian subgroup satisfying the
condition $(\ast)$. In the case of $L_{s}$ is simply connected, due to $L_s$ has a finite abelian subgroup satisfying
the condition $(\ast)$ the root system of $L$ is one of the types $D_4$, $D_5$, $\E_6$. In the $D_5$ case, $L_s$
neither have finite abelian subgroup satisfying the condition $(\ast)$ by \cite{Yu3}, Proposition 2.2. Thus $L$
is conjugate to one of $L_1$, $L_2$, $L_3$, $L_4$, $L_5$, $L_6$ listed ahead of this Lemma. In the case of $L_1$
is of type $3A_1$, one shows $F\sim F_{13}$ easily. The $D_4$ and $D_6$ are considered in Propositions 2.2 and 2.3
of \cite{Yu3}. The $\E_6$ is considered in Section 2. The $D_4+A_1$ and $D_5+A_1$ can be answered by
some argument as the arguments in \cite{Yu3}, Section 2.
\end{proof}


\begin{prop}\label{P:Weyl-E7}
There are exact sequences \[1\longrightarrow\Hom(\mathbb{F}_2^2,\mathbb{F}_2^2)\rtimes\GL(2,\mathbb{F}_2)
\longrightarrow W(F_{13})\longrightarrow W(\F_4)\longrightarrow 1,\]
\[1\longrightarrow\Hom(\mathbb{F}_{2}^{3},\mathbb{F}_{2}^{2})\rtimes\GL(3,\mathbb{F}_2)\longrightarrow W(F_{14})
\longrightarrow \{\pm{1}\}^{3}\rtimes S_3\longrightarrow 1,\]
\[1\longrightarrow(\mathbb{F}_{2}^{4}\times\mathbb{F}_{2})\rtimes(\Hom(\mathbb{F}_{2}^{2},
\mathbb{F}_{2}^{2})\rtimes\GL(2,\mathbb{F}_2))\longrightarrow W(F_{15})\longrightarrow D_4\longrightarrow 1,\]
\[1\longrightarrow S_4\times(\{\pm{1}\}^{3}\rtimes\GL(3,\mathbb{F}_2))\longrightarrow W(F_{16})\longrightarrow
\{\pm{1}\}\longrightarrow 1,\]
\[1\longrightarrow Q_1\longrightarrow W(F_{17})\longrightarrow\{\pm{1}\}\longrightarrow 1,\]
\[1\longrightarrow\mathbb{F}_2^{5}\rtimes(\mathbb{F}_2^{4}\rtimes S_6)\longrightarrow W(F_{18})
\longrightarrow\{\pm{1}\}\longrightarrow 1,\]
\[1\longrightarrow\mathbb{F}_2^{5}\rtimes((\mathbb{F}_2^{4}\rtimes S_6)\times S_2)\longrightarrow W(F_{19})
\longrightarrow\{\pm{1}\}\longrightarrow 1,\]
\[1\longrightarrow\mathbb{F}_3^{3}\rtimes\SL(3,\mathbb{F}_3)\longrightarrow W(F_{20})\longrightarrow\{\pm{1}\}
\longrightarrow 1,\] where $Q_1$ is the Weyl group of the subgroup $F_1$ of $\Spin(12)/\langle c\rangle$
defined in \cite{Yu3}, Proposition 2.5, which is of order $3\times 2^{12}$; in the exact sequences for
$W(F_{18})$ and $W(F_{19})$, $\mathbb{F}_2^{4}$ means
$\{\vec{x}\in\mathbb{F}_2^{6}: \sum_{1\leq i\leq 6} x_{i}=0\}/\mathbb{F}_2\cdot(1,1,1,1,1,1)$,
$\mathbb{F}_2^{5}$ means $\{\vec{x}\in\mathbb{F}_2^{6}: \sum_{1\leq i\leq 6} x_{i}=0\}$.
\end{prop}

\begin{proof}
Since $Z(L_1)_0$ is conjugate to a maximal torus of the derived subgroup $(L_2)_{s}$ of the Levi subgroup $L_2$.
We have $(L_2)_{s}\cong\Spin(8)$ and $W(Z(L_1)_0)\cong W(\F_4)$. There are two ways to show this: one is to use
the automorphism $\eta$ of $\mathfrak{\fre_6}$ as defined in Section 2, another is to embed $(L_2)_{s}$ into a
subgroup isomorphic to $\F_4$. There is a homomorphism $p: W(F_{13})\longrightarrow W(Z(L_1)_0)$, which is
apparently a surjective map. Since $L_1=C_{G}(Z(L_1)_0)$, consideration in $(L_1)_{s}$ shows
$\ker p\cong\Hom(\mathbb{F}_2^2,\mathbb{F}_2^2)\rtimes\GL(2,\mathbb{F}_2)$, where the first $\mathbb{F}_2^2$
means $F_{13}/Z(L_1)_0$, the second $\mathbb{F}_2^2$ means $F_{13}\cap Z(L_1)_{s}$, and $\GL(2,\mathbb{F}_2)$ acts
on the first $\mathbb{F}_2^2$ linearly and faithfully and acts on the second one trivially. The determination of
the other Weyl groups is along the same line by first identifying $W(Z(L_{i})_0)$ and then calculating $\ker p$.
The surjection of $p$ is clear. For $F_{14}$, $F_{17}$, $F_{18}$, $F_{19}$, we use results in \cite{Yu3},
Section 2. For $F_{20}$, we use Proposition \ref{P:Weyl-E6-F2 and F3}. For $F_{16}$, we note that both of
$$\langle[(-1,1)],[(e_1e_2e_3e_4,1)],[(e_1e_2e_5e_6,1)],[(e_1e_3e_5e_7,1)]\rangle$$ and
$$\langle[(-1,1)],[(e_8e_9,\textbf{i})],[(e_8e_{10},\textbf{j})]\rangle$$ are stable under the action of $\ker p$.
For $F_{15}$, we note that $$K=\langle[(-1,1)],[(c,1)],[(e_1e_2e_3e_4,1)],[(e_1e_2e_5e_6,1)]\rangle$$ stable under
the action of $\ker p$. The induced map $p': \ker p\longrightarrow W(K)$ is surjective and one has
$$W(K)\cong\Hom(\mathbb{F}_{2}^{2},\mathbb{F}_{2}^{2})\rtimes\GL(2,\mathbb{F}_2),$$ where
$\GL(2,\mathbb{F}_2)$ acts on the first $\mathbb{F}_2^2$ linearly and faithfully and acts on the second one
trivially. Moreover by choosing generators of $W(K)$ appropriately one sees that $W(K)\subset W(F_{15})$. Since
$$C_{(L_3)_{s}}(K)=\langle[(e_1e_3e_5e_7,1)],\Spin(2)^{4}\times\Sp(1)\rangle,$$ one can show that
$\ker p'=(\mathbb{F}_{2})^{4}\times\mathbb{F}_{2}$, which the $\mathbb{F}_{2}^{4}$ factor means the action on
$[(e_1e_3e_5e_7,\textbf{i})]$ and the factor $\mathbb{F}_{2}$ means the action on $[(e_1e_2,\textbf{j})]$.
\end{proof}

\section{Simple Lie group of type $\bf E_8$}

Let $G=\E_8$. It contains a subgroup $H_1$ isomorphic to $$(\SU(5)\times\SU(5))/\langle(\omega_5 I,\omega_5^{2}I)
\rangle.$$ In $H_1$ let $F_1=\langle[(\omega_5 I,I)],[A_5,B_5],[A_5,B_5^{2}]\rangle$, where
$$A_5=\diag\{1,e^{\frac{2\pi i}{5}},e^{\frac{4\pi i}{5}},e^{\frac{6\pi i}{5}},e^{\frac{8\pi i}{5}}\},\quad
B_5=\left(\begin{array}{ccccc}0&1&0&0&0\\0&0&1&0&0\\0&0&0&1&0\\0&0&0&0&1\\1&0&0&0&0\\\end{array}\right).$$

\begin{lemma}\label{L:E8-order}
Let $F$ be a finite abelian subgroup of $G$ satisfying the condition $(*)$. If $|F|$ has a prime factor $p\geq 5$,
then $F\sim F_1$.
\end{lemma}

\begin{proof}
Choosing an element $x\in F$ with $o(x)=p\geq 5$, then $F\subset G^{x}$ and $G^{x}$ is connected by Steinberg's
theorem. Hence $G^{x}$ is semisimple as $F$ is a finite abelian subgroup of $G$ satisfying the condition $(*)$.
Therefore the root system of $\frg^{x}$ is of type $2A_4$ (cf. \cite{Oshima}) and
$G^{x}\cong(\SU(5)\times\SU(5))/\langle(\omega_5 I,\omega_5^{2}I)\rangle$. Hence $F\sim F_1$.
\end{proof}

\begin{prop}\label{P:Weyl-E8-F1}
We have $W(F_1)\cong\SL(3,\mathbb{F}_5)$.
\end{prop}

\begin{proof}
From the the construction of $F_1$ in $(\SU(5)\times\SU(5))/\langle(\omega_5 I,\omega_5^{2}I)\rangle$, one has
$\Stab_{W(F_1)}(x)\cong\Hom((\mathbb{F}_5)^2,\mathbb{F}_5)\rtimes\SL(2,\mathbb{F}_5)$. On the other hand $W(F_1)$
acts transitively on $F_1-\{1\}$. Therefore $W(F_1)\cong\SL(3,\mathbb{F}_5)$.
\end{proof}

The group $G$ has a subgroup $H_2$ isomorphic to $(\E_6\times\SU(3))/\langle(c,\omega I)\rangle$, where $c$ is
an element of order $3$ generating $Z(\E_6)$. From Section 2, one sees that the group
$\E_6/\langle c\rangle$ has two conjugacy classes of finite abelian subgroups satisfying the condition $(\ast)$
and with a nontrivial bimultiplicative function $m$, with representatives $F_1$ and $F_3$ in the terminology
there. From them we have finite abelian subgroups $F_2$, $F_3$ of $G$ satisfying the condition $(\ast)$. One has
$F_2\cong(C_2)^{3}\times(C_3)^{3}$ and $F_3\cong(C_3)^{5}$.

\begin{prop}\label{P:E8-Non two}
If $F$ is a finite abelian subgroup of $G$ satisfying the condition $(*)$ and with $|F|$ being a multiple of $3$,
then $F\sim F_2$ or $F_3$.
\end{prop}

\begin{proof}
Choosing an element $x\in F$ of order $3$, then $F\subset G^{x}$ and $G^{x}$ is connected by
Steinberg's theorem. Hence $G^{x}$ is semisimple as $F$ is a finite abelian subgroup of $G$
satisfying the condition $(*)$. Therefore the root system of $\frg^{x}$ is of type $A_8$ or
$E_6+A_2$ (cf. \cite{Oshima}). In the case of the root system of $\frg^{x}$ is of type $E_6+A_2$,
one has $F\sim F_2$ or $F_3$ by Propositions \ref{P:E6-inner finite 1} and \ref{P:E6-inner finite}.
In the case of the root system of $\frg^{x}$ is of type $A_8$, one has $G^{x}\cong\SU(9)/\langle
\omega I\rangle$. In this case the conjugacy class of $F$ is unique and it contains an element
conjugate to $[\diag\{A_3,A_3,A_3\}]$. However the element $[\diag\{A_3,A_3,A_3\}]$ is conjugate
to $\exp(\frac{2\pi i}{3}(H'_1+H'_5+H'_8))$ in $G$ and which in turn is conjugate to the
$\exp(\frac{2\pi i}{3}(H'_1-H'_3))$ giving a centralizer group of type $E_6+A_2$. Therefore this
conjugacy class is the one containing $F_3$.
\end{proof}

\begin{prop}
We have $$W(F_2)\cong\SL(3,\mathbb{F}_3)\times\SL(3,\mathbb{F}_2),$$ $$W(F_3)\cong\mathbb{F}_{3}^{4}
\rtimes\GSp(2,\mathbb{F}_3),$$ where $\GSp(2,\mathbb{F}_3)$ means the group of $4\times 4$ general
symplectic matrices over the field $\mathbb{F}_3$.
\end{prop}

\begin{proof}
For $F_2$, by the construction of $F_2$ in $(\E_6\times\SU(3))/\langle(c,\omega I)\rangle$ and
Proposition \ref{P:Weyl-E6 inner finite 1}, one has $\Stab_{W(F_2)}(x)\cong(\mathbb{F}_{3}^{2}
\rtimes\SL(2,\mathbb{F}_3))\times\SL(3,\mathbb{F}_2)$, where we note that it is $\SL(2,\mathbb{F}_3)$
instead of $\GL(2,\mathbb{F}_3)$ since it is $\E_6$, not the disconnected group $\Aut(\fre_6)$ being
considered here. Since $W(F_2)$ acts transitively on the set of elements of $F_2$ of order $3$, one has
$W(F_2)\cong\SL(3,\mathbb{F}_3)\times\SL(3,\mathbb{F}_2)$.

Denote by $\theta_1=\exp(\frac{2\pi i(H'_1+H'_2-H'_3)}{3})$, $\theta_2=\exp(\frac{2\pi i(H'_1-H'_3)}{3})$.
Then, the root system of $\frg^{\theta_1}$ is of type $A_8$ and the root system of $\frg^{\theta_2}$ is of
type $E_6+A_2$. In $G^{\theta_1}\cong\SU(9)/\langle\omega I\rangle$, $F_3$ is conjugate to
$$F'_3=\langle[\omega I],[A'_3],[B'_3],[\diag\{A_3,A_3,A_3\}],[\diag\{B_3,B_3,B_3\}]\rangle,$$ where
$$A_3=\diag\{1,\omega,\omega^2\}, A'_3=\diag\{I_3,\omega I_3,\omega^{2} I_3\},$$
$$B_3=\left(\begin{array}{ccc}0&1&0\\0&0&1\\1&0&0\\\end{array}\right),$$
$$B'_3=\left(\begin{array}{ccc}0_3&I_3&0_3\\0_3&0_3&I_3\\I_3&0_3&0_3\\\end{array}\right).$$ Denote by
$$K=\langle[A'_3],[B'_3],[\diag\{A_3,A_3,A_3\}],[\diag\{B_3,B_3,B_3\}]\rangle.$$ Then, for any element
$x\in K$, one has $x\sim[A'_3]$, $\theta_1 x\sim\theta_1[A'_3]$, $(\theta_1)^{2}x\sim(\theta_1)^{2}[A'_3]$.
In $G$, one can show that $[A'_3]\sim\theta_2$ and $\theta_1 [A'_3]\sim(\theta_1)^{2}[A'_3]\sim\theta_1$.
That means $K$ is stable under the action of $W(F'_3)$. Note that there is a nondegenerate symplectic
form on $K$ taking values in the field $\mathbb{F}_3$ from the inclusion $K\subset G^{\sigma_1}\cong
\SU(9)/\langle\omega I\rangle $. Considertation in $G^{\theta_1}$ shows that
$\Stab_{W(F_3)}=\Sp(2,\mathbb{F}_3)$. By this $|W(F_{3})|=2\times 3^{4}\times|\Sp(2,\mathbb{F}_3)|$.
Also, we note that $G$ contains an element $g$ normalizing $\SU(9)/\langle\omega\rangle$ and acting on it
as $g[X]g^{-1}=[\overline{X}]$, $X\in\SU(9)$. Therefore $g\in N_{G}(F'_3)$ which acts on $K$ as a
transformation in $\GSp(2,\mathbb{F}_3)$ and outside $\Sp(2,\mathbb{F}_3)$, maps $\theta_1$ to $\theta_1^{-1}$.
There is a homomorphism $p: W(F'_{13})\longrightarrow W(K)$. We show that $W(K)\cong\GSp(2,\mathbb{F}_3)$.
Since by the above $\GSp(2,\mathbb{F}_3)\subset\Im p$ and $|W(F_{3})|=3^{4}\times|\GSp(2,\mathbb{F}_3)|$,
this in turn implies $W(F_3)\cong\mathbb{F}_{3}^{4}\rtimes\GSp(2,\mathbb{F}_3).$ To prove $W(K)\cong
\GSp(2,\mathbb{F}_3)$, note that each elementary abelian 3-subgroup of rank 2 with degenerate (or
nondegenerate) symplectic form is conjugate to $S_1=\langle[A'_3],[\diag\{A_3,A_3,A_3\}]\rangle$ (or
$S_2=\langle[A'_3],[B'_3]\rangle$). In Section 2, we define the subgroup $F_3$ of $\Aut(\fre_6)$ as
\[F_3=\langle\pi(\theta'_1),\eta,[(A_3,A_3,A_3)],[(B_3,B_3,B_3)]\rangle,\] where $\theta'_1$ means the
element $\theta_1$ there. Thus in $G^{\theta_2}\cong(\E_6\times\SU(3))/\langle(c,\omega I)\rangle$,
\[F_3=\langle[(c,1)],[(\theta'_1,A_3)],[(\eta,B_3)],[((A_3,A_3,A_3),I)],[((B_3,B_3,B_3),I)]\rangle.\]
By calculation one shows that the elements $[((A_3,A_3,A_3),I)]$, $[(\eta,B_3)]$ are conjugate to
$\theta_2$ in $G$. Denote by $$K'=\langle[(c,1)],[(\eta,B_3)],[((A_3,A_3,A_3),I)],[((B_3,B_3,B_3),I)]
\rangle,$$ $S'_1=\langle[(c,1)],[((A_3,A_3,A_3),I)]\rangle$, and $S'_2=\langle[(c,1)],[(\eta,B_3)]
\rangle$. Then, $K'$ is conjugate to $K$, and any rank 2 elementary abelian 3-subgroup of $F_3$ is
conjugate to $S'_1$ or $S'_2$. Thus $\{S_1,S_2\}$, $\{S'_1,S'_2\}$ are pairwise conjugate. As
$\frg^{S'_2}=\mathfrak{so}(8)+\mathbb{C}^{4}\not\cong\frg^{S'_1}=\mathfrak{su}(3)^{4}$, $S_1$ and $S_2$
are not conjugate in $G$. Then, the action of $W(K)$ on $K$ preserves the symplectic form up to scaler.
By it one has $W(K)\subset\GSp(2,\mathbb{F}_3)$. It is showed above that $W(K)\supset\GSp(2,\mathbb{F}_3)$.
Thus $W(K)|=|\GSp(2,\mathbb{F}_3)$.
\end{proof}

\begin{lemma}\label{L:E8-order2}
If $F$ is a finite abelian 2-subgroup of $G$ satisfying the condition $(*)$, then $x^4=1$ for any $x\in F$.
\end{lemma}

\begin{proof}
Suppose the conclusion does not hold. Then $F$ contains an element $x$ of order $8$. One has $F\subset G^{x}$
and $G^{x}$ is connected by Steinberg's theorem. Hence $G^{x}$ is semisimple. Since $x\in Z(G^{x})$ is of
order $8$. The root system of $\frg^{x}$ is of type $A_7+A_1$. Even in this case
$G^{x}\cong(\SU(8)\times\Sp(1))/\langle (-I),(iI,-1)\rangle$ and hence $o(x)\neq 8$.
\end{proof}

In \cite{Huang-Yu}, we defined a Klein four subgroup $\Gamma_1$ of $G$ and one has \[G^{\Gamma_1}\cong
(\E_6\times\U(1)\times\U(1))/\langle(c,\omega,1)\rangle\rtimes\langle z\rangle,\] where
$(\fre_6\oplus i\bbR\oplus i\bbR)^{z}=\frf_4\oplus 0\oplus 0$. Given a finite abelian subgroup $F'$ of
$\E_6\rtimes\langle z\rangle$ satisfying the condition $(\ast)$, $F=F'\times\Gamma_1$ is a finite abelian
subgroup of $G$ satisfying the condition $(\ast)$. Let $F_4$, $F_5$, $F_6$, $F_7$, $F_8$, $F_9$ be finite
abelian subgroups of $G$ obtained in this way from finite abelian subgroups $F_6$, $F_8$, $F_9$, $F_{10}$,
$F_{11}$, $F_{12}$ of $\Aut(\fre_6)$ as defined in Section 3.

\begin{prop}\label{P:E8-two-3}
If $F$ is a finite abelian 2-subgroups of $G$ satisfying the condition $(*)$ and containing a Klein four
subgroup conjugate to $\Gamma_1$, then it is conjugate to one of $F_4$, $F_5$, $F_6$, $F_7$, $F_8$, $F_9$.
\end{prop}

\begin{proof}
We may and do assume that $\Gamma_1\subset F$. Then, \[F\subset G^{\Gamma_1}\cong(\E_6\times\U(1)\times
\U(1))/\langle(c,\omega,1)\rangle\rtimes\langle z\rangle.\] Hence $F\not\subset(G^{\Gamma_1})_0$. By this,
for any $[(x,\lambda_1,\lambda_2)]\in F\cap(G^{\Gamma_1})_0$, one has $[(x,\lambda_1,\lambda_2)]$ being
conjugate to $[(\phi(x),\lambda_1^{-1},\lambda_2^{-1})]$, where $\phi$ is an outer automorphism of $\E_6$.
Therefore $\lambda_1,\lambda_2\in\{\pm{1}\}$. Thus $F$ is conjugate to a subgroup of the form
$F=F'\times\Gamma_1$, where $F'$ is a subgroup of$\E_6\rtimes\langle z\rangle$ with $F'\not\subset\E_6$
and satisfying the condition $(*)$. By Lemma \ref{L:E6-outer finite1} and Proposition \ref{P:E6-outer finite2},
$F$ is conjugate to one of $F_4$, $F_5$, $F_6$, $F_7$, $F_8$, $F_9$.
\end{proof}

Among $F_4$, $F_5$, $F_6$, $F_7$, $F_8$, $F_9$, the maximal abelian ones are $F_4$, $F_5$, $F_8$, $F_9$.

\begin{prop}\label{P:Weyl-E8-3}
There is an exact sequence \[1\rightarrow\Hom'((\mathbb{F}_2)^3,(\mathbb{F}_2)^3)\rightarrow W(F_4)
\rightarrow\Hom(\bbF_2^{3},\bbF_2^{2})\rtimes(\GL(2,\bbF_2)\times\GL(2,\bbF_2))
\rightarrow 1,\] and we have \[W(F_5)\cong\Hom(\bbF_{2}^{6},\bbF_{2}^{2})\rtimes\big(\GL(2,\bbF_2)\times
((\GL(3,\bbF_2)\times\GL(3,\bbF_2))\rtimes S_2)\big),\]
\[W(F_6)\cong\bbF_2^{6}\rtimes\Sp(2,2;0,0),\]
\[W(F_7)\cong\bbF_2^{7}\rtimes\Sp(1,3;0,0),\]
\[|W(F_{8})|=3^{2}\times 2^{20},\]
\[W(F_9)\cong\bbF_2^{8}\rtimes\Sp(0,4;0,0),\] where the group $\Sp(r,s;\epsilon,\delta)$ is defined in \cite{Yu},
Page 259.
\end{prop}

\begin{proof}
The abelian subgroups $F_5$, $F_6$, $F_7$, $F_9$ are elementary abelian 2-subgroups. The description of their Weyl
groups is given in \cite{Yu}, Proposition 8.17. For $F_4$, let $K=\{x\in F_2: x^2=1\}$ and $K'=\{x^2: x\in F_2\}$.
Then, $K'\subset K$, $\rank K=5$, and $\rank K'=3$. From the description of the subgoup $F_6$ in \cite{Yu3},
one sees that $K$ is conjugate to the elementary abelian $2$-subgroup $F'_{2,1}$ of $G$ in the terminology of
\cite{Yu}. Hence $W(K)\cong\Sp(2,1;1,0)\cong\Hom(\bbF_2^{3},\bbF_2^{2})\rtimes(\GL(2,\bbF_2)\times\GL(2,\bbF_2))$
by \cite{Yu} Proposition 7.26, where the first $\GL(2,\bbF)$ acts on $\bbF_2^{3}$ as a subgroup of $\GL(3,\bbF)$
and the second $\GL(2,\bbF)$ acts on $\bbF_2^{2}$. There is a homomorphism $p: W(F_4)\longrightarrow W(K)$. Using
the fact that $W(F_2)$ acts transitively on the set of pairs $(x_1,x_2)$ of elements of $F_4$ generating a Klein
four subgroup conjugate to $\Gamma_1$, one gets that $p$ is surjective. Similarly as the proof of
\cite{Yu3}, Proposition 3.1, one shows that $\ker p\cong\Hom'((\mathbb{F}_2)^3,(\mathbb{F}_2)^3$. For $F_8$, let
$K=\{x\in F_{8}: x^2=1\}$ and $K'=\{x^2: x\in F_{8}\}$. Then, $K'\cong\mathbb{F}_2$. Denote by $z$ a generator of
$K'$. Then $W(F_8)$ fixes $z$ and stabilizes $K$. Hence there is a homomorphism $p: W(F)\longrightarrow
\Stab_{W(K)}(z)$. Similar argument as \cite{Yu3}, Proposition 3.3 shows that $\ker p\cong(\mathbb{F}_2)^2$ and $p$
is surjective. Since $K'\sim F_{6}$, the orbit $W(K)z$ has three elements. One has $|W(K)|=|W(F_6)|=64\times 256
\times 6^3\times 2=3^3\times 2^{18}$ and hence $|W(F_{6})|=\frac{4}{3}|W(K)|=3^{2}\times 2^{20}$.
\end{proof}

\begin{example}
Denote by $F=F_5$, which is the elementary abelian 2-subgroup $F_{3,3}$ in the terminology of \cite{Yu}. Recall
that $F$ possesses a decomposition $F=A\times B\times B'$, where $A$ is a pure $\sigma_2$ elementary abelain
2-subgroup of rank 2, $B_1$, $B_2$ are pure $\sigma_1$ elementary abelain 2-subgroups of rank 3. Moreover, if we
define a function $\mu: F\rightarrow\{\pm{1}\}$ by $\mu(x)=-1$ if $x\sim\sigma_1$ and $\mu(x)=1$ otherwise, then
$\mu(xx_1x_2)=\mu(x)\mu(x_1)\mu(x_2)$ for any $x\in A$, $x_1\in B_1$, $x_2\in B_2$. We may assume that
$\sigma_1\in F$. Then, $F\subset G^{\sigma_1}\cong\Spin(16)/\langle c\rangle$. In $G^{\sigma_1}$, we have
\begin{eqnarray*}F&=&\langle[-1],[e_1e_2e_3e_4e_5e_6e_7e_8],[e_1e_2e_3e_4],[e_1e_2e_5e_6],[e_1e_3e_5e_7],\\&&
[e_{9}e_{10}e_{11}e_{12}],[e_{9}e_{10}e_{13}e_{14}],[e_{9}e_{11}e_{13}e_{15}]\rangle,\end{eqnarray*} where
$$A=\langle[-1],[e_1e_2e_3e_4e_5e_6e_7e_8]\rangle,$$ $$B_1=\langle[e_1e_2e_3e_4],[e_1e_2e_5e_6],[e_1e_3e_5e_7]
\rangle,$$ $$B_2=\langle[e_{9}e_{10}e_{11}e_{12}],[e_{9}e_{10}e_{13}e_{14}],[e_{9}e_{11}e_{13}e_{15}]\rangle.$$
For the fine group grading of $\fre_8(\bbC)$ associated to $F$, given a root $\alpha$, if $\alpha|_{A}=1$, then
$\alpha|_{B_1}=1$ or $\alpha|_{B_2}=1$ by the above description of $F$ in $G^{\sigma_1}$; if $\alpha|_{A}\neq 1$,
by the property of $\mu$, there are subgroups $B'_1$ of $B_2$, and $B'_2$ of $B_2$ such that
$F=A\times B'_1\times B'_2$ has similar property as the decomposition $F=A\times B_1\times B_2$ regarding
the function $\mu$, and $\alpha|_{B'_1}=\alpha|_{B'_2}=1$. In either case, we find that the induced action of
$s_{\alpha,\xi}$ on $F/A$ stabilizes both $B_1A/A$ and $B_2A/A$, which implies that
\[W_{small}(F)\subset\Hom(\bbF_{2}^{6},\bbF_{2}^{2})\rtimes\big(\GL(2,\bbF_2)\times\GL(3,\bbF_2)\times
\GL(3,\bbF_2)\big),\] where $W_{small}(F)$ is the small Weyl group associated to the fine group grading
corresponding to $F$ and $s_{\alpha,\xi}$ is transvection defined in \cite{Han-Vogan}. Therefore
$W_{small}(F)\neq W(F)$ in this example.
\end{example}

For the involution $\sigma_1$ of $G$, one has (cf. \cite{Huang-Yu}, Table 2) $$G^{\sigma_1}\cong(\E_7\times
\Sp(1))/\langle(c,-1)\rangle.$$ There are elements $\eta_4,\omega'\in \E_7$ (cf. \cite{Yu}, Sections 7 and 8)
with $\eta_4^2=c$, $\omega'=c$, $\eta_4\omega'\eta_4^{-1}\omega'^{-1}=c$,
$$\fre_7^{\eta_4}\cong\fre_7^{\omega'}\cong\fre_7^{\eta_4\omega'}\cong\mathfrak{su}(8)$$ and
$$\E_7^{\langle\eta_4,\omega'\rangle}\cong\SO(8)/\langle-I\rangle.$$ In $G^{\sigma_1}$,  let
\[F_{10}=\langle\sigma_1,[(\eta_4,\textbf{i})],[(\omega',\textbf{j})],[I_{4,4}],[\diag\{I_{2,2},I_{2,2}\}],
[\diag\{I_{1,1},I_{1,1},I_{1,1},I_{1,1}\}]\rangle,\]
\begin{eqnarray*}F_{11}&=&\langle\sigma_1,[(\eta_4,\textbf{i})],[(\omega',\textbf{j})],[I_{4,4}],
[\diag\{I_{2,2},I_{2,2}\}],\\&&[\diag\{I_{1,1},I_{1,1},I_{1,1},I_{1,1}\}],[\diag\{J'_{1},J'_{1},J_{1},J_{1}\}]
\rangle.\end{eqnarray*} Neither $F_{10}$ nor $F_{11}$ is a maximal abelian subgroup.

\begin{prop}\label{P:E8-two-1}
If $F$ is a finite abelian 2-subgroup of $G$ satisfying the condition $(*)$, containing an element conjugate to
$\sigma_1$ and no Klein four subgroups conjugate to $\Gamma_1$, then $F$ is conjugate to $F_{10}$ or $F_{11}$.
\end{prop}

\begin{proof}
We may and do assume that $\sigma_1\in F$. Then \[F\subset G^{\sigma_1}\cong(\E_7\times\Sp(1))/\langle(c,-1)
\rangle.\] Apparently the image of the projection of $F$ to $\Sp(1))/\langle-1\rangle$ is conjugate to
$\langle\textbf{i}),\textbf{j})\rangle$. Hence the image $F'$ of the projection of $F$ to $\E_7/\langle c\rangle$
has a nontrivial bimultiplicative function, i.e., it lifts to a non-abelian subgroup of $\E_7$. The group $F'$
is an abelian 2-subgroup of $\E_7/\langle c\rangle$ satisfying the condition $(*)$ and containing no elements
conjugate to $\sigma_2$ in the terminology of Section 2. By Propositions \ref{P:E7-finite5} and \ref{P:E7-finite3},
$F'$ is conjugate to one of $F_8$, $F_9$, $F_{10}$, $F_{11}$, $F_{12}$ in the terminology of Section 2. As
we assume that $F$ contains no Klein four subgroups conjugate to $\Gamma_1$, the subgroup $\{x\in F: x^2=1\}$ is
conjugate to one of $F''_{r,s}: r\leq 3,s\leq 2$ by \cite{Yu}, Proposition 8.11. The bimultipliclative function
on $F_{12}$ is trivial, and the maximal elementary abelian 2-subgroups of the abelian 2-subgroups of $G$ constructed
from $F_{8}$ or $F_9$ is not conjugate to any of $F''_{r,s}: r\leq 3,s\leq 2$, as they contain Klein four subgroups
conjugate to $\Gamma_1$. Therefore $F'$ is conjugate to $F_{10}$, $F_{11}$ in Section 2 and hence $F$ is conjugate
to $F_{10}$ or $F_{11}$.
\end{proof}

There is an involution $\sigma_2$ of $G$ with $G^{\sigma_2}\cong\Spin(16)/\langle c\rangle$ (cf. \cite{Huang-Yu},
Table 2). In $G^{\sigma_2}$, let
\begin{eqnarray*} F_{12}&=&\langle -1,c,e_1e_2e_3e_4e_5e_6e_7e_8,e_1e_2e_3e_4e_{9}e_{10}e_{11}e_{12},
e_1e_2e_5e_6e_{9}e_{10}e_{13}e_{14},\\&& e_1e_3e_5e_7e_{9}e_{11}e_{13}e_{15}\rangle,\end{eqnarray*}
\begin{eqnarray*} F_{13}=&=&\langle -1,c,e_1e_2e_3e_4e_5e_6e_7e_8,e_1e_2e_3e_4e_{9}e_{10}e_{11}e_{12},
e_1e_2e_5e_6e_{9}e_{10}e_{13}e_{14},\\&& e_1e_3e_5e_7e_{9}e_{11}e_{13}e_{15},e_1e_2e_5e_7e_9e_{12}\rangle.
\end{eqnarray*} None of them is a maximal abelian subgroup of $G$.

\begin{prop}\label{P:E8-two-2}
If $F$ is a finite abelian 2-subgroup of $G$ satisfying the condition $(*)$ and containing no elements conjugate
to $\sigma_1$, then $F$ is conjugate to $F_{12}$ or $F_{13}$.
\end{prop}

\begin{proof}
In the case of $F$ is an elementary abelian 2-subgroup, one has $F\sim F_{12}$ by \cite{Yu}, Proposition 8.11.
In the case of $F$ is not an elementary abelian 2-subgroup, we may assume that $\sigma_2\in F$ and there is
an element $x\in F$ with $x^2=\sigma_2$. Then, $F\subset G^{\sigma_2}\cong\Spin(16)/\langle c\rangle$ and
$x\sim [e_1e_2e_3e_4e_5e_6]$ in $\Spin(16)/\langle c\rangle$. Therfore $F\subset G^{x}\cong(\Spin(6)\times
\Spin(10))/\langle(c_6,c_{10})\rangle$. If there are no elements $x=[x'],y=[y']\in F$, $x',y'\in\Spin(16)$ with
$x'y'x'^{-1}y'^{-1}=c$, then $F$ lifts to an abelian 2-subgroup of $\Spin(16)$. In this case one can show that
$F\sim F_{13}$. If there are elements $x=[x'],y=[y']\in F$, $x',y'\in\Spin(16)$ with $x'y'x'^{-1}y'^{-1}=c$,
let $F'_1$, $F'_2$ be the images of the projections of $F$ to $\SO(6)/\langle -I\rangle$ and
$\SO(10)/\langle -I\rangle$ respectively. By \cite{Yu2}, Proposition 3.1, it is associated integers
$k_1,k_2,s_{0,1},s_{0,2}\geq 0$ with $6=2^{k_1}\cdot s_{0,1}$ and $10=2^{k_2}\cdot s_{0,2}$. Then, $k_1,k_2=0$
or $1$. Since $x'y'x'^{-1}y'^{-1}=c$, one has $k_1=k_2=1$. In this case $F'$ contains an element conjugate to
$[(\diag\{I_{1,1},J'_1,J_1\},\diag\{I_{1,1},I_{1,1},I_{1,1},J'_1,J_1\})]$ and hence $F$ contains an element
conjugate to $$y=[e_1\frac{e_3+e_4}{\sqrt{2}}\frac{1+e_5e_6}{\sqrt{2}}e_{7}e_{9}e_{11}
\frac{e_{13}+e_{14}}{\sqrt{2}}\frac{1+e_{15}e_{16}}{\sqrt{2}}].$$ But $y^2=[-e_{5}e_{6}e_{15}e_{16}]$. It leads to
a contradiction since $[-e_{5}e_{6}e_{15}e_{16}]$ is conjugate to $\sigma_1$ in $G$.
\end{proof}





\smallskip

Let $F\subset G$ be an abelian subgroup satisfying the condition $(*)$ and $\fra=\Lie F\otimes_{\bbR}\bbC\subset
\frg$. Denote by $L=C_{G}(\fra)$ and $L_s=[L,L]$. By Lemma \ref{L:center}, $F=Z(L)_0\cdot(F\cap L_s)$ and
$F':=F\cap L_s$ is a finite abelian subgroup of $L_{s}$ satisfying the condition $(*)$. In $G$, let $L_1$,
$L_2$, $L_3$, $L_4$, $L_5$ be Levi subgroups of $G$ with root systems $D_4$, $D_6$, $D_7$, $E_6$, $E_7$
respectively. For each of them, one has $Z((L_{i})_s)\subset Z(L_{i})_0$.

In $L_1$, let
$$F_{14}=\langle Z(L_2)_0,e_1e_2e_3e_4,e_1e_2e_5e_6,e_1e_3e_5e_7\rangle.$$ In $L_2$, let
$$F_{15}=\langle Z(L_2)_0,-1,c,e_1e_2e_3e_4,e_1e_2e_5e_6,e_1e_2e_7e_8,e_1e_2e_{9}e_{10},e_1e_3e_5e_7e_9e_{11}
\rangle.$$
In $L_3$, let \begin{eqnarray*}F_{16}&=&\langle Z(L_3)_0,-1,c,e_1e_2e_3e_4,e_1e_2e_5e_6,e_1e_3e_5e_7,
e_8e_9e_{10}e_{11},\\&&e_8e_9e_{12}e_{13},e_8e_{10}e_{12}e_{14}\rangle.\end{eqnarray*}
In $L_4$, let $F_{17}=\langle Z(L_4)_0,F'\rangle$, where $F'$ projects to the subgroup $F_2$ of $\Aut(\fre_6)$
in Section 2. In $L_5$, let $F_{18}=\langle Z(L_4)_0,F'_{18}\rangle$ and $F_{19}=\langle Z(L_4)_0,F'_{19}\rangle$,
where $F'_{18}$, $F'_{19}$ are finite abelian subgroups of $\E_7$ with images of projection to $\Aut(\fre_7)$
being the finite abelian subgroups $F_9$, $F_{12}$ of $\Aut(\fre_7)$ in Section 2.

\begin{lemma}\label{L:E8-nonfinite}
Let $F$ be a non-finite closed abelian subgroup of $G$ satisfying the condition $(*)$. Then $F$ is either a
maximal torus, or is conjugate to one of $F_{14}$, $F_{15}$, $F_{16}$, $F_{17}$, $F_{18}$, $F_{19}$.
\end{lemma}

\begin{proof}
By Lemma \ref{L:Levi}, $L_s$ is simply connected. As it has a finite abelian subgroup satisfying the
condition $(\ast)$, the root system of $L$ is one of the types $\emptyset$, $D_4$, $D_6$, $D_7$, $E_6$
or $E_7$. In the case of $D_4$, $D_6$, $D_7$, $F$ is conjugate to $F_{14}$, $F_{15}$, $F_{16}$ by \cite{Yu3},
Propositions 2.2 and 2.3. The $E_6$ case is answered in Section 2, and the $E_7$ case is answered in
Section 4.
\end{proof}

\begin{prop}\label{P:Weyl-E8}
There are exact sequences \[1\longrightarrow\Hom((\mathbb{F}_{2})^{3}, (\mathbb{F}_{2})^{2})\rtimes
\GL(3,\mathbb{F}_2)\longrightarrow W(F_{14})\longrightarrow W(\F_4)\longrightarrow 1,\]
\[1\longrightarrow Q_2\longrightarrow W(F_{15})\longrightarrow D_4\longrightarrow 1,\]
\[1\longrightarrow((\mathbb{F}_{2}^{3}\rtimes\GL(3,\mathbb{F}_2))\times(\mathbb{F}_{2}^{3}\rtimes
\GL(3,\mathbb{F}_2)))\rtimes S_2\longrightarrow W(F_{16})\longrightarrow\{\pm{1}\}\longrightarrow 1,\]
\[1\longrightarrow\mathbb{F}_{3}^{3}\rtimes\SL(3,\mathbb{F}_3)\longrightarrow W(F_{17})\longrightarrow
D_6 \longrightarrow 1,\]
\[1\longrightarrow\Sp(3,\mathbb{F}_2)\longrightarrow W(F_{18})\longrightarrow\{\pm{1}\}\longrightarrow 1,\]
\[1\longrightarrow Q_3\longrightarrow W(F_{19})\longrightarrow\{\pm{1}\}\longrightarrow 1.\]

Here $Q_2$ is the Weyl group of the subgroup $F_{12}$ of $\Spin(12)$ defined in \cite{Yu3}, Proposition 2.2,
which is of order $5\times 3^2\times 2^{14}$ and fits in an exact sequence $$1\rightarrow\Hom(\mathbb{F}_2,
\mathbb{F}_2^{6})\rightarrow Q_2\rightarrow(\Hom(\mathbb{F}_2^{4},\mathbb{F}_2)\rtimes S_6\rightarrow 1.$$
The group $Q_3$ is the Weyl group of a subgroup $F'_{12}$ of $\E_7$, which is the preimage in $\E_7$ of the
subgroup $F_{12}$ of $\Aut(\fre_7)$ defined in Section 2. By Proposition \ref{P:Weyl-E7-3C_4}, there is an exact
sequence $$1\rightarrow\mathbb{F}_2^3\times\Hom'(\mathbb{F}_2^3,\mathbb{F}_2^3)\rightarrow Q_3\rightarrow
\GL(3,\mathbb{F}_2)\rightarrow 1,$$ where $\Hom'(\mathbb{F}_2^3,\mathbb{F}_2^3)=\{f\in\Hom(\mathbb{F}_2^3,
\mathbb{F}_2^3): \tr f=0\}$.
\end{prop}

\begin{proof}
In each case we have a homomorphism $p: W(F)\longrightarrow W(Z(L)_0)$, which is clearly a surjective map. The
neutral component of the center of $L_1$, $L_2$, $L_3$, $L_4$, $L_5$ is a maximal torus of the derived subgroup of
a Levi subgroup with root system $D_4$, $2A_1$, $A_2$, $A_1$, $A_1$ respectively. by embedding the $D_4$, $A_2$
subgroups into $E_6$, one shows that $W(Z(L_1)_0)\cong W(\F_4)$ and $W(Z(L_4)_0)\cong D_6$. One has
$\ker p=W(F\cap L_s)$ for $F\cap L_s$ as a subgroup of $L_s$, which is determined in \cite{Yu3} and previous
sections of this paper.
\end{proof}

Jun Yu \\ School of Mathematics, \\ Institute for Advanced Study, \\ Einstein Drive, Fuld Hall, \\
Princeton, NJ 08540, USA \\
email:junyu@math.ias.edu.

\end{document}